\documentclass[11pt,reqno]{amsart}

\RequirePackage{amssymb}
\RequirePackage{amsmath}
\RequirePackage{amsthm}
\RequirePackage{array}
\RequirePackage{cancel}
\RequirePackage{enumerate}
\RequirePackage{mathtools}
\RequirePackage{tikz-cd}
\usetikzlibrary{positioning}
\usepackage{dynkin-diagrams} 
\RequirePackage{geometry}
\RequirePackage{comment}
\usepackage[colorlinks = true,
            linkcolor = blue,
            urlcolor  = blue,
            citecolor = blue]{hyperref}
\allowdisplaybreaks

\newcommand{\R}{\mathbb R}
\newcommand{\C}{\mathbb C}

\newcommand{\g}{\mathfrak g}
\renewcommand{\k}{\mathfrak k}
\newcommand{\p}{\mathfrak p}
\renewcommand{\t}{\mathfrak t}
\renewcommand{\a}{\mathfrak a}
\newcommand{\s}{\mathfrak s}

\renewcommand{\r}{\mathfrak r}
\renewcommand{\l}{\mathfrak l}
\newcommand{\q}{\mathfrak q}
\newcommand{\f}{\mathfrak f}
\newcommand{\h}{\mathfrak h}
\renewcommand{\sl}{\mathfrak {sl}}
\renewcommand{\sp}{\mathfrak {sp}}
\newcommand{\so}{\mathfrak {so}}
\newcommand{\su}{\mathfrak {su}}
\newcommand{\D}{\mathcal {D}}
\newcommand{\V}{\mathcal {V}}
\newcommand{\M}{\mathcal {M}}
\newcommand{\N}{\mathcal {N}}
\newcommand{\U}{\mathcal {U}}
\newcommand{\CC}{\mathcal {C}}
\renewcommand{\P}{\mathcal {P}}
\newcommand{\RR}{\mathcal R}

\renewcommand{\i}{\operatorname{i}}
\newcommand{\Ad}{\operatorname{Ad}}
\newcommand{\ad}{\operatorname{ad}}
\newcommand{\supp}{\operatorname{supp}}
\newcommand{\Id}{\operatorname{Id}}

\renewcommand{\Re}{\operatorname{Re}}
\renewcommand{\Im}{\operatorname{Im}}
\renewcommand{\c}{\operatorname{c}}
\newcommand{\nc}{\operatorname{nc}}
\newcommand{\spann}{\operatorname{span}}
\newcommand{\Hom}{\operatorname{Hom}}

\newcommand{\Gr}{\operatorname{Gr}}

\theoremstyle{plain}
\newtheorem{maintheorem}{Theorem}

\newtheorem{theorem}{Theorem}[section]
\newtheorem{proposition}[theorem]{Proposition}
\newtheorem{lemma}[theorem]{Lemma}
\newtheorem{corollary}[theorem]{Corollary}
\newtheorem{conjecture}[theorem]{Conjecture}

\theoremstyle{definition}
\newtheorem{example}[theorem]{Example}
\newtheorem{examples}[theorem]{Examples}
\newtheorem{algorithm}[theorem]{Algorithm}

\theoremstyle{remark}
\newtheorem{remark}[theorem]{Remark}
\newtheorem{definition}[theorem]{Definition}

\newcommand{\TwoRowDynkinA}[2]{%
\begin{tikzpicture}[>=Stealth, line cap=round, line join=round]
  \def\toppattern{#1}
  \def\bottompattern{#2}
  \StrLen{\toppattern}[\N]
  \def\dx{1}
  \def\dy{1}
  \def\r{3pt}

  \foreach \i in {1,...,\N}{
    \StrChar{\toppattern}{\i}[\nodechar]
    \if\nodechar*
      \node[circle, fill, inner sep=0pt, minimum size=2*\r] (T\i) at ({(\i-1)*\dx}, \dy) {};
    \else
      \node[circle, draw, inner sep=0pt, minimum size=2*\r] (T\i) at ({(\i-1)*\dx}, \dy) {};
    \fi
  }

  \foreach \i in {1,...,\N}{
    \StrChar{\bottompattern}{\i}[\nodechar]
    \if\nodechar*
      \node[circle, fill, inner sep=0pt, minimum size=2*\r] (B\i) at ({(\i-1)*\dx}, 0) {};
    \else
      \node[circle, draw, inner sep=0pt, minimum size=2*\r] (B\i) at ({(\i-1)*\dx}, 0) {};
    \fi
  }

  \pgfmathtruncatemacro{\Nmone}{\N-1}
  \foreach \i in {1,...,\Nmone}{
    \pgfmathtruncatemacro{\j}{\i+1}
    \draw (T\i) -- (T\j);
    \draw (B\i) -- (B\j);
  }

  \foreach \i in {1,...,\N}{
    \draw[<->, shorten <=2pt, shorten >=2pt] (T\i.south) -- (B\i.north);
  }
\end{tikzpicture}%
}

\title{Balanced and pluriclosed metrics on real  semisimple Lie groups}
\author{Joseph Kwong}
\begin{document}

\begin{abstract}
    We characterise the existence of balanced and pluriclosed metrics on compact quotients of real semisimple Lie groups equipped with regular complex structures,  in terms of Vogan diagrams.  Consequently,  such complex manifolds cannot simultaneously admit  a balanced metric and a pluriclosed metric. Along the way, we revisit and correct the classification of regular complex structures on real semisimple Lie groups.
\end{abstract}

\maketitle

\setcounter{tocdepth}{1}
\tableofcontents

\section{Introduction}
A \textit{Hermitian metric} on a complex manifold $(M^{2n},J)$ is a Riemannian metric $g$ on $M$ such that $J^* g = g$. A Hermitian metric $g$ is \textit{balanced} if the $2$-form $\omega := g(J \cdot,\cdot)$ satisfies 
$d \omega^{n-1} = 0,$  and \textit{pluriclosed} if $\partial \overline \partial \omega = 0$. 
We say that a complex manifold is \textit{balanced}/\textit{pluriclosed} if it admits a balanced/pluriclosed metric, respectively. Every K\"ahler manifold is both balanced and pluriclosed. In fact, a Hermitian metric is K\"ahler if and only if it is simultaneously balanced and pluriclosed; see
\cite{ivanov} or \cite{popovici_aeppli}. The \textit{Fino--Vezzoni conjecture} predicts the converse at the manifold level:
\begin{conjecture}[Fino--Vezzoni {\cite{fino_conjecture_statement}}]
    A compact complex manifold admitting both a balanced metric and a (possibly distinct) pluriclosed metric also admits a K\"ahler metric.
\end{conjecture}

The main aim of this article is to study the conjecture on Lie group quotients. Let $G$ be a Lie group, let $\Gamma$ be a cocompact lattice (i.e. a discrete subgroup such that $\Gamma \backslash G$ is compact), and let $J$ be a left-invariant complex structure on $G$. Then $J$ descends to a complex structure on $\Gamma \backslash G$, and we obtain a compact complex manifold $(\Gamma \backslash G,J)$. 
The Fino--Vezzoni conjecture is known to hold for $(\Gamma \backslash G,J)$ when $J$ is a bi-invariant complex structure \cite{wang},  when $G$ is a nilpotent Lie group \cite{fino_nilmanifolds}, and for other classes of solvable Lie groups \cite{fino_conjecture_statement, otiman, fino_2023, Zheng_2024, Fribert_2025, fino_2025}.
Our main result verifies the conjecture when $G$ is any semisimple Lie group and $J$ is any \textit{regular} left-invariant complex structure:
\begin{maintheorem}
    \label{main theorem}
    Let $G$ be a real semisimple Lie group, let $J$ be a regular complex structure on $G$, and let $\Gamma$ be a cocompact lattice of $G$. The compact complex manifold $(\Gamma \backslash G, J)$ cannot admit both a balanced metric and  a pluriclosed metric.
\end{maintheorem}

Recall that every semisimple Lie group $G$ admits a cocompact lattice \cite{borel}.
 A left-invariant complex structure $J$ on $G$ is called \textit{regular} if it is 
also right-invariant with respect to a Cartan subgroup of $G$ \cite{snow_1986}.  

Theorem \ref{main theorem} has already been proved when $G$ is compact semisimple: in this case, $(G, J)$ does not admit any balanced metrics \cite[Corollary 5.1]{fino_2019}. Moreover, the left-invariant complex structures admitting pluriclosed metrics on compact $G$ are classified by Lauret and Montedoro in  \cite{lauret}. 
For each  even-dimensional non-compact simple $G$ of \textit{inner type} (i.e. there is a Cartan involution of the Lie algebra of $G$ which is an inner automorphism), Giusti and Podest\`a \cite{giusti_podesta_2023} construct a family of regular complex structures $J$ for which $(\Gamma \backslash G, J)$ admits balanced metrics but no pluriclosed metrics.

\subsection{Refined statements} 
Using a well-known symmetrisation argument (see Proposition \ref{symmetrisation 1 prop}), Theorem \ref{main theorem} follows from  Theorems \ref{balanced technical theorem} and \ref{pluriclosed technical theorem}, which completely characterise the existence of  balanced and pluriclosed metrics on regular complex structures using \textit{Vogan diagrams}. We now introduce some notation and definitions in order to state these technical results. We refer the reader to \S \ref{semisimple Lie algebras section} for the background on real semisimple Lie algebras and Vogan diagrams.

Let $\g$ denote the Lie algebra of the semisimple Lie group $G$, and let $\g^\C$ denote its complexification. A left-invariant complex structure $J$ on $G$ is determined by the $\i$-eigenspace $\q$ of $J:\g^\C \rightarrow \g^\C$ (see \S \ref{invariant complex structures subsection}). We say that $J$ is a \textit{regular complex structure} if $J \circ \ad(H) = \ad(H) \circ J$ for all $H \in \h$ for some Cartan subalgebra $\h$ of $\g$ (see \S \ref{regular classification subsection}). In this case, we say that $J$ is \textit{$\h$-regular}. Every regular complex structure on $G$ can be constructed by the following algorithm, which is a corrected version of a procedure in \cite{snow_1986}  (see Algorithm \ref{regular complex structures construction algorithm}):
\begin{enumerate}[(i)]
    \item Choose a maximally compact Cartan subalgebra $\h$ of $\g$. The complex conjugation map of $\g^\C$ induces an automorphism $\sigma$ of the root system $\Phi$ of $(\g^\C, \h^\C)$.
    \item Choose a set of positive roots $\Phi^+$ of $\Phi$ such that $\sigma \Phi^+ = - \Phi^+$.
    \item Choose a subset $\Delta_0$ of the simple roots of $\Phi^+$ such that no vertex of $\Delta_0$ is adjacent to a vertex of  $-\sigma \Delta_0$ in the Dynkin diagram determined by $\Phi^+$.
    \item Choose a complex subspace $\l$ of $\h^\C$ such that $\l \oplus \sigma \l = \h^\C$ and $[\g^\C_\alpha, \g^\C_{-\alpha}] \subseteq \l$ for all $\alpha \in \Delta_0$.
\end{enumerate}
Set $\RR_0 := \R \Delta_0 \cap  \Phi$, and $\RR := \RR_0 \sqcup \Phi^+ \backslash (\RR_0^+ \sqcup (-\sigma \RR_0^+))$. Then
$$\q := \l \oplus \bigoplus_{\alpha \in \RR} \g^\C_\alpha$$
is the $\i$-eigenspace of a regular complex structure $J = J(\Phi^+, \Delta_0, \l)$, where  $\g^\C_\alpha$ denotes the root space of $\alpha \in \Phi$.

A \textit{Vogan diagram} is a triple $\V = (\D, \theta, \P)$, where $\D$ is a Dynkin diagram, $\theta$ is an involutive automorphism of $\D$ whose two-cycles are denoted by bidirectional arrows, and $\P$ is a subset of the fixed points of $\theta$, which we paint. 

Fix a maximally compact Cartan subalgebra $\h$ of $\g$. Let $\Phi$ denote the roots of $(\g^\C, \h^\C)$. Let $J= J(\Phi^+, \Delta_{0}, \mathfrak l)$ be an $\h$-regular complex structure on $G$. The choice of positive roots $\Phi^+$ determines a Vogan diagram $\V$ (see Algorithm \ref{semisimple g to Vogan diagram}).

\begin{maintheorem}
    \label{balanced technical theorem}
    Let $G$ be a semisimple Lie group, and let $J$ be an $\h$-regular complex structure on $G$.
    There exists a left-invariant balanced metric on $(G,J)$ if and only if none of the connected components of $\V$ are listed in Table  \ref{vogan table}.
\end{maintheorem}

When $G$ is compact semisimple, every connected component of $\V$ is listed in Table \ref{vogan table}, so $(G,J)$ is never balanced (cf. \cite[Corollary 5.1]{fino_2019}). When $G$ is inner simple, \cite[Lemma 3.2]{giusti_podesta_2023} construct their regular complex structures $J$ so that $\V$ has at least two painted vertices. Then $\V$ is not listed in Table \ref{vogan table}, so $(\Gamma \backslash G,J)$ is balanced.

Let $\g_1,\ldots, \g_s$ denote the simple factors of $\g$, and let $B_i$ denote the Killing form of $\g_i$. When $\g$ is inner, there is a unique Cartan involution $\theta: \g \rightarrow \g$  such that $\theta \h = \h$ (see Remark \ref{cartan involution inner}).  In this case, set  $\theta_i := \theta|_{\g_i}: \g_i \rightarrow \g_i$.

\begin{maintheorem}
    \label{pluriclosed technical theorem}
    Let $G$ be a semisimple Lie group, and let $J$ be an $\h$-regular complex structure on $G$.
    \begin{enumerate}[\normalfont (a)]
        \item If $\g$ is not inner, then there are no left-invariant pluriclosed metrics on $(G,J)$.
        \item If $\g$ is inner, then the following are equivalent:
        \begin{enumerate}[\normalfont (i)]
            \item $(G,J)$ admits a left-invariant pluriclosed metric.
            \item Every connected component of $\V$ is listed in Table \ref{vogan table}, and the inner product $-(\sum_{i=1}^s \kappa_i B_i)|_{\h \times \h}$ is $J|_{\h}$-invariant for some $\kappa_1,\ldots, \kappa_s > 0$.
            \item  The left-invariant metric $g = -\sum_{i=1}^s \kappa_i B_i(\cdot, \theta_i \cdot) $ is a pluriclosed metric on $(G,J)$ for some $\kappa_1,\ldots, \kappa_s > 0$.
        \end{enumerate}
    \end{enumerate}
\end{maintheorem}

For an even-dimensional compact semisimple $G$, every connected component of $\V$ is listed in Table \ref{vogan table}, so there exist regular $J$ such that $(G,J)$ is pluriclosed. Since $\theta = \Id$,  the metrics  $-\sum \kappa_i B_i(\cdot, \cdot) $ are precisely the bi-invariant metrics on $G$. In particular, if $(G,J)$ is pluriclosed, then there exists a bi-invariant pluriclosed metric (cf. \cite{lauret}). 

For each even-dimensional inner semisimple $G$, \cite{giusti_podesta_2023} construct their regular complex structures $J$ so that $\V$ has at least two painted vertices. Then $\V$ is not listed in Table \ref{vogan table}, so $(\Gamma \backslash G,J)$ is not pluriclosed.

\begin{table}[h]
\centering
\begin{tabular}{c|c}
    Family  & Vogan Diagram \\ \hline
    Compact & Any Vogan diagram with no 2-cycles and no painted vertices \\
    $A_r$, $r \geq 1$  & Any $A_r$ diagram with no 2-cycles and exactly one painted vertex \\
    $B_r$, $ r \geq 2$ & \dynkin[edge length=1cm, root radius=.1cm]{B}{*o.oo} \\
    $C_r$, $ r \geq 3$ & \dynkin[edge length=1cm, root radius=.1cm]{C}{oo.oo*} \\
    $D_r$, $ r \geq 4$ & \dynkin[edge length=1cm, root radius=.1cm]{D}{*o.ooo} and \dynkin[edge length=1cm, root radius=.1cm]{D}{oo.o*o} \\
    $E_6$, $E_7$ & \dynkin[edge length=1cm, root radius=.1cm]{E}{ooooo*} and \dynkin[edge length=1cm, root radius=.1cm]{E}{oooooo*} \\
\end{tabular}
\caption{Connected Vogan diagrams appearing in Theorems \ref{balanced technical theorem} and \ref{pluriclosed technical theorem}.}
\label{vogan table}
\end{table}

\begin{corollary}
    \label{main corollary}
    Let $G$ be an even-dimensional semisimple Lie group, and let $\Gamma$ be a cocompact lattice.
    \begin{enumerate}[\normalfont (i)]
        \item There exists a regular $J$ such that $(\Gamma \backslash G, J)$ is balanced if and only if no simple factor of $\g$ is $\sl(2,\R)$ or compact.
        \item There exists a regular $J$ such that $(\Gamma \backslash G, J)$ is pluriclosed if and only if every simple factor of $\g$ is compact or Hermitian.
        \item If $\g$ is simple, then there exists a regular $J$ such that $(\Gamma \backslash G, J)$ is either balanced or pluriclosed.
    \end{enumerate}
\end{corollary}

Let $\g$ be a simple Lie algebra, and let $G$ denote the adjoint group of $\g$. We say that $\g$ is \textit{compact} if $G$ is compact. We say that $\g$ is \textit{Hermitian} if $G/ K$ is a Hermitian symmetric space, where $K$ is a maximal compact subgroup of $G$. In the notation of \cite{knapp}, the simple Hermitian Lie algebras are precisely $\su(p,q)$, $\mathfrak{so}(2,q)$, $\mathfrak{sp}(2n,\R)$, $\mathfrak{so}^*(2n)$, $E\text{III}$, and $E\text{VII}$ (cf. \cite[page 315]{besse}). These Lie algebras are all inner, but not every inner $\g$ is compact or Hermitian (e.g. $\g = \so(4,1)$).

\begin{examples}
    Let $G$ be a semisimple Lie group, and let $\Gamma$ be a cocompact lattice. The following statements follow immediately from Corollary \ref{main corollary}.
    \begin{enumerate}[\normalfont (i)]
        \item Let $\g = \sl(2,\R) \oplus \sl(2,\R)$. Since $\sl(2,\R) \cong \su(1,1)$ is Hermitian, there exist regular complex structures $J$ such that $(\Gamma \backslash G, J)$ is pluriclosed, but $(\Gamma \backslash G, J)$ is not balanced for any regular $J$.
        \item Let $\g = \su(2,1)$. There exist regular $J_1$ and $J_2$ such that $(\Gamma \backslash G, J_1)$ is balanced and  $(\Gamma \backslash G, J_2)$ is pluriclosed: this follows from the fact that $\su(2,1)$ has two Vogan diagrams, one of which is listed in Table \ref{vogan table} and one of which is not.
        \item Suppose $\g$ is the underlying real Lie algebra of a complex semisimple Lie algebra. Then no  simple factor of $\g$ is inner, so for any regular $J$,  $(\Gamma \backslash G, J)$ is balanced but not pluriclosed.
    \end{enumerate}
\end{examples}

\begin{table}[h]
\centering
\begin{tabular}{c|c|c|c|c|c}
    $\g$ & Vogan diagrams & Inner? & Pluriclosed? & Balanced? & $\dim_\R \g$ \\\hline
    $\sl(2,\C) \cong \so(3,1)$  & \begin{dynkinDiagram}[edge/.style={draw=none},edge length=1cm, root radius=.1cm]{A}{oo}\draw[<->,shorten <=2pt,shorten >=2pt] (root 1) -- (root 2);\end{dynkinDiagram} & no & none & all & 6 \\
    $\su(2) \oplus \su(2)$ & \dynkin[edge/.style={draw=none},edge length=.5cm, root radius=.1cm]{A}{oo} & yes & some & none  & 6 \\
    $\su(2) \oplus \sl(2,\R)$ & \dynkin[edge/.style={draw=none},edge length=.5cm, root radius=.1cm]{A}{o*} & yes & some & none  & 6 \\
    $\sl(2,\R) \oplus \sl(2,\R)$ & \dynkin[edge/.style={draw=none},edge length=.5cm, root radius=.1cm]{A}{**} & yes & some & none  & 6 \\
    $\su(3)$ & \dynkin[edge length=.5cm, root radius=.1cm]{A}{oo} & yes & some & none & 8 \\
    $\su(2,1)$ & \dynkin[edge length=.5cm, root radius=.1cm]{A}{**}, \dynkin[edge length=.5cm, root radius=.1cm]{A}{*o} & yes & some & some & 8 \\
    $\sl(3,\R)$ & \dynkin[edge length=1cm, root radius=.1cm, involutions = {12}]{A}{oo} & no & none & all & 8 \\
    $\so(5) \cong \sp(2)$  & \dynkin[edge length=.5cm, root radius=.1cm]{C}{oo} & yes & some  & none & 10 \\
    $\so(4,1) \cong \sp(1,1)$ & \dynkin[edge length=.5cm, root radius=.1cm]{C}{*o} & yes & none & all  & 10\\
    $\so(3,2) \cong \sp(4,\R)$ & \dynkin[edge length=.5cm, root radius=.1cm]{C}{**}, \dynkin[edge length=.5cm, root radius=.1cm]{C}{o*} & yes & some & some & 10 \\
    $\g_2^{\c}$ & \dynkin[edge length=.5cm, root radius=.1cm]{G}{oo} & yes & some & none & 14 \\
    $\g_2^{\nc}$ & \dynkin[edge length=.5cm, root radius=.1cm]{G}{**}, \dynkin[edge length=.5cm, root radius=.1cm]{G}{*o}, \dynkin[edge length=.5cm, root radius=.1cm]{G}{o*} & yes & none & all & 14 
\end{tabular}
\caption{All real semisimple Lie algebras with $\text{rank}( \g^\C) = 2$ are listed, along with their Vogan diagrams. The Lie algebras $\g_2^{\c}$ and $\g_2^{\nc}$ denote the compact and non-compact real forms of the complex Lie algebra $G_2$, respectively. The fourth and fifth columns indicate whether, among regular complex structures $J$ on $G$, none, some, or all yield pluriclosed or balanced manifolds $(\Gamma \backslash G, J)$, respectively.}
\label{rank 2 table}
\end{table}

\subsection{Article outline and proof strategy}
We recall the basics of left-invariant complex structures on Lie groups in \S \ref{invariant complex structures subsection}. In \S\S \ref{symmeterisation 1 subsection} and \ref{symmeterisation 2 subsection}, we state and prove the symmetrisation procedures for balanced and pluriclosed metrics. 

\S \ref{semisimple Lie algebras section} treats several aspects of real semisimple Lie algebras needed later: Theorems \ref{balanced technical theorem} and \ref{pluriclosed technical theorem} are stated using maximally compact Cartan subalgebras and Vogan diagrams, so we study these in \S\S \ref{Cartan subalgebras subsection}--\ref{maximally compact subsection} and \S\ref{Vogan diagrams subsection}, respectively. The proofs of Theorems \ref{balanced technical theorem} and \ref{pluriclosed technical theorem} rely on a characterisation of the Vogan diagrams in Table \ref{vogan table} (\S \ref{table 1 subsection}), a normalisation of the root vectors (\S \ref{root vectors subsection}), and a description of left-invariant right $T$-invariant metrics on $G$ (\S \ref{right T invariant subsection}), where $T$ is the connected subgroup with Lie algebra $\t$ given in (\ref{t and a definitions}).

We study regular complex structures in  \S \ref{regular complex structures section}. To the best of the author's understanding, the classification of regular complex structures in \cite{snow_1986} is not complete (see Remark \ref{counterexample for snow}). We state and prove  a complete classification of regular complex structures on semisimple Lie groups in \S\S \ref{regular classification subsection} and  \ref{regular structure proof}, respectively. 

We prove Theorems \ref{balanced technical theorem} and \ref{pluriclosed technical theorem} in \S\S \ref{balanced section} and \ref{pluriclosed section}, respectively. By symmetrisation (see Proposition \ref{symmetrizing to right T invariant} and Lemma \ref{AdGT is compact}), it suffices to consider the balanced and pluriclosed equations for left-invariant right $T$-invariant Hermitian metrics on $(G,J)$.
The balanced equation for these metrics is given in Proposition \ref{balanced equation prop}. This equation generalises \cite[Equation 2.3]{giusti_podesta_2023}, which treats the inner case. In \S \ref{balanced pairs subsection}, we determine the solvability of this equation using root system arguments and the Vogan diagram classification from \S \ref{table 1 subsection}.

When $\g$ is not inner, we show that $(G,J)$ does not admit any left-invariant pluriclosed metrics in Proposition \ref{not inner pluriclosed}.  When $\g$ is inner, the pluriclosed equation for left-invariant right $T$-invariant Hermitian metrics on $(G,J)$ is given in Proposition \ref{pluriclosed metrics description}. This equation generalises \cite[Theorem 1.1]{lauret}, which treats the compact case. The solvability of this equation is then determined directly in \S \ref{pluriclosed section}.

\section{Invariant Hermitian structures on Lie groups}
\label{invariant Hermitian section}
\subsection{Invariant complex structures on Lie groups} 
\label{invariant complex structures subsection}Let $G$ be a real connected Lie group. A complex structure $J$ on $G$ is called \textit{left-invariant} if the left translation maps $L_a:G \rightarrow G$ given by $L_a(b) = ab$ are all holomorphic with respect to $J$.

Let $\g$ denote the Lie algebra of $G$, i.e. the left-invariant vector fields on $G$.  If $J$ is a left-invariant complex structure on $G$, then $JX \in \g$ for all $X \in \g$. Thus, $J$ is determined by a  linear map $J:\g \rightarrow \g$. Let $\g^\C = \g \oplus \hat \i \g$ denote the complexification of $\g$, and let $\sigma:\g^\C \rightarrow \g^\C$ be the complex conjugation map $\sigma: X + \hat \i Y \mapsto X - \hat \i Y$ for $X,Y \in \g$. Complexifying $J$, we obtain a $\C$-linear map $J:\g^\C \rightarrow \g^\C$. Consider the $\i$-eigenspace of  $J:\g^\C \rightarrow \g^\C$, which we denote by $\q$. The $-\i$-eigenspace is $\sigma \q$, and we have $\q \oplus \sigma \q = \g^\C$. Moreover, the integrability of $J$ implies that $\q$ is a subalgebra of $\g^\C$. Conversely, every subalgebra $\q$ of $\g^\C$ satisfying $\q \oplus \sigma \q = \g^\C$ arises in this manner. In summary, we have a bijection
\begin{equation}
    \label{invariant complex structures bijection}
    \left\{ \begin{matrix}
    \text{Left-invariant complex } \\ \text{structures $J$ on $G$}
\end{matrix}\right\} \longleftrightarrow\left\{ \begin{matrix}
    \text{Subalgebras $\q$ of $\g^\C$ } \\ \text{such that $\q \oplus \sigma\q = \g^\C$}
\end{matrix}\right\}.
\end{equation}

\begin{example}
    \label{bi invariant complex structure 1}
    A complex structure on $G$ is called \textit{bi-invariant} if both the left and right translation maps are holomorphic. 
    If $J$ is bi-invariant, then we can form a complex Lie algebra $\s$ by equipping $\g$ with $\i \cdot X := JX$. 
    Conversely, if $\s$ is a complex Lie algebra and $\g$ is the underlying real Lie algebra of $\s$, then $J: \g \rightarrow \g$, $X \mapsto \i X$ is a bi-invariant complex structure on $G$. 
    
    Suppose $\s$ is a complex Lie algebra and $\g$ is the underlying real Lie algebra. The following is an isomorphism of complex Lie algebras:
    $$\g^\C \rightarrow \s \oplus \overline \s, \qquad X + \hat \i Y \mapsto (X + \i Y, X - \i Y), \quad X,Y \in \g.$$
    Here, $\overline \s$ is the complex Lie algebra with underlying real Lie algebra $\g$ and scalar multiplication $\lambda \cdot X := \overline \lambda X$, for $\lambda \in \C$ and $X \in \g$, where $\overline \lambda X$ denotes scalar multiplication in $\s$. We identify $\g^\C \cong \s \oplus \overline \s$ via the isomorphism above. Under this identification,  complex conjugation $\sigma$ is given by $\sigma:(X,Y) \mapsto (Y,X)$ for $X,Y \in \g$.
    
    Now, let $J: \g \rightarrow \g$ be the bi-invariant complex structure on $G$ given by $J X  = \i X$. The complexification $J:\g^\C \rightarrow \g^\C$ is given by $J(X,Y) = (\i X, \i Y)$.
    Thus, the $\pm \i$-eigenspaces of $J$ are $\q = \s \oplus 0$ and $\sigma \q = 0 \oplus \overline \s$, respectively.
\end{example}

We say that two left-invariant complex structures $J_1$ and $J_2$ on $G$ are \textit{equivalent} if there exists a Lie algebra automorphism $f: \g \rightarrow \g$ such that $f \circ J_1 = J_2 \circ f$. If $\q_i \subseteq \g^\C$ denotes the $\i$-eigenspace of $J_i: \g^\C \rightarrow \g^\C$, then $J_1$ and $J_2$ are equivalent if and only if there exists a (real) Lie algebra automorphism $f: \g \rightarrow \g$ such that $f \q_1 = \q_2$.

\subsection{Symmetrising Hermitian metrics on $\Gamma \backslash G$}
\label{symmeterisation 1 subsection}

The following result is well-known to experts, and is stated and proved under additional assumptions in \cite[\S 2]{fino_grant_2004}, \cite[Prop 3.6]{ugarte}, \cite[Theorem 7]{belgun} and  \cite[Propositions 2.1 and 2.3]{giusti_podesta_2023}. Although the proof is similar, we include it here for the reader's convenience.
\begin{proposition}
    \label{symmetrisation 1 prop}
    Let $G$ be a Lie group, and let $\Gamma$ be a cocompact lattice. Let $J$ be a left-invariant complex structure on $G$, and let $J$ denote the induced complex structure on $\Gamma \backslash G$. The following are equivalent:
    \begin{enumerate}[\normalfont (i)]
        \item $(\Gamma \backslash G,J)$ admits a K\"ahler/balanced/pluriclosed metric.
        \item $(G,J)$ admits a left-invariant K\"ahler/balanced/pluriclosed metric.
    \end{enumerate} 
\end{proposition}

\begin{proof}
     Let $\pi:G \rightarrow M = \Gamma \backslash G$ denote the natural projection map. For each $X \in \g$, there exists a unique smooth vector field $X' \in \mathfrak{X}(M)$ which is $\pi$-related to $X$. 
    Since $G$ admits a cocompact lattice, $G$ is unimodular \cite[Lemma 6.2]{milnor}; by definition, this means that there exists  a nowhere vanishing bi-invariant volume form $\mu$ on $G$. Since $\mu$ is left-invariant, it descends to a nowhere-vanishing form $\mu'$ on $M$.
    For each $k$, define
    $$S: \Omega^k(M) \rightarrow \Lambda^k \g^*, \qquad (S \alpha)(X_1,\ldots, X_k) := \int_M \alpha(X_1',\ldots, X_k') \, \mu',$$
    where  $\Omega^k(M)$ denotes the smooth $k$-forms $\alpha$ on $M$, and $\Lambda^k \g^*$ denotes the left-invariant $k$-forms on $G$. 
    The map $S$ satisfies the following properties: \begin{enumerate}[\normalfont (a)]
        \item $S$ commutes with the exterior derivative $d$.
        \item Extending $S$ to complex $k$-forms in the obvious way, $S$ sends $(p,q)$-forms to $(p,q)$-forms.
        \item $S$ commutes with the Dolbeault operators $\partial$ and $\overline \partial$.
        \item $S$ sends positive $(p,p)$-forms to positive $(p,p)$-forms.
    \end{enumerate}
    We show (a) only; the other properties are easy to show. 
    Let $\alpha$ be a $k$-form on $M$, and let $X_0,\ldots, X_k \in \g$.  We find
    \begin{align*}
        (d S \alpha)(X_0,\ldots, X_k) &= \sum_{i <j} (-1)^{i+j} (S \alpha)([X_i, X_j],X_0,\ldots, \widehat{X_i},\ldots, \widehat{X_j},\ldots, X_k) \\
        &=\int_M \sum_{i <j} (-1)^{i+j} \alpha([X_i', X_j'],X_0',\ldots, \widehat{X_i'},\ldots, \widehat{X_j'},\ldots, X_k') \, \mu' \\
        &= \int_M (d \alpha)(X_0',\ldots, X_k')\, \mu' - \sum_i (-1)^i \int_M X'_i \alpha(X_0',\ldots, \widehat{X_i'},\ldots, X'_k) \, \mu' \\
        &=\int_M (d \alpha)(X_0',\ldots, X_k')\, \mu' \\
        &=  (S d \alpha )(X_0,\ldots, X_k),
    \end{align*}
    where a hat denotes the omission of that element. The second last equality holds from the following fact:
    if $X \in \g$ and $f:M \rightarrow \R$ is a smooth function, then
    \begin{equation}
        \label{int vanish lemma}
        \int_M X'(f )\, \mu' = 0.
    \end{equation}
    Let us prove Equation (\ref{int vanish lemma}). By Cartan's magic formula, we find that 
        $$\int_M \mathcal L_{X'
        }(f \mu') = \int_M (d  \circ \iota_{X'} )(f \mu') + \int_M( \iota_{X'} \circ d)(f \mu') = 0,$$
        where $\mathcal L$ is the Lie derivative, and $\iota$ denotes interior product. The last equality follows from Stokes' theorem and that $f \mu'$ is a top-degree form. The product rule implies that 
        $$0 = \int_M \mathcal L_{X'
        }(f \mu') =\int_M X'(f )\, \mu' + \int_M f\, \mathcal L_{X'}\mu' =  \int_M X'(f )\, \mu' .$$
        The last equality follows because $\mathcal L_{X'}\mu' = 0$:   the flow of $X \in \g$ is given by the family of right translations $t \mapsto R_{\exp(tX)}$, and $\mu$ is right-invariant, so $\pi^* (\mathcal L_{X'
        } \mu') = \mathcal L_X \mu  = 0$.
    
    Now, let $\omega'$ be a positive $(1,1)$-form on $(M,J)$.
    By (a) and (c), if $\omega'$ is K\"ahler or pluriclosed, respectively, then so is $\omega := S \omega'$. 
    Next, suppose $d (\omega')^{n-1} = 0$. By \cite[\S 4]{michelsohn_1982}, the map $\alpha \mapsto \alpha^{n-1}$ is a bijection (on any complex manifold) between positive $(1,1)$-forms and positive $(n-1,n-1)$-forms. Thus, $\Omega' := (\omega')^{n-1}$ is a positive $(n-1,n-1)$-form with $d \Omega' = 0$. By (d) and (a), $\Omega := S \Omega'$ is a left-invariant positive $(n-1,n-1)$-form on $(G,J)$ with $d \Omega = 0$. There exists a unique positive $(1,1)$-form $\omega$ such that $\omega^{n-1} = \Omega$. Uniqueness of $\omega$ implies that $\omega$ is also left-invariant. 
\end{proof}

\subsection{Symmetrising invariant Hermitian metrics on $G$}
\label{symmeterisation 2 subsection} Let $G$ be a Lie group and let $T$ be a connected Lie subgroup of $G$. Let $\g$ and $\t$ be the Lie algebras of $G$ and $T$, respectively.  The following  are equivalent for a left-invariant complex structure $J$ on $G$:
\begin{enumerate}[(i)]
    \item $J$ is \textit{right $T$-invariant}, i.e. the diffeomorphism $R_a: G \rightarrow G$ given by $R_a(b) := ba$ is holomorphic for any $a \in T$.
    \item $J:\g \rightarrow \g$ is \textit{$\Ad_G(T)$-invariant}, i.e. $\Ad(a) \circ J  = J \circ \Ad(a)$ for all $a \in T$. 
    \item $J:\g \rightarrow \g$ is \textit{$\ad(\t)$-invariant}, i.e. $\ad(X) \circ J = J \circ \ad(X)$ for all $X \in \t$. 
    \item The $\i$-eigenspace $\q$ of $J: \g^\C \rightarrow \g^\C$ is \textit{$\ad (\t^\C)$-invariant}, i.e. $[\t^\C, \q] \subseteq \q$.
\end{enumerate}

 The following are equivalent for a left-invariant covariant tensor field $\alpha$ on $G$:
\begin{enumerate}[(i)]
    \item $\alpha$ is \textit{ right $T$-invariant}, i.e. $(R_a)^* \alpha = \alpha$ for all $a \in T$.
    \item $\alpha: \g\times \cdots \times \g \rightarrow \R$ is $\Ad_G(T)$-invariant, i.e. $\Ad(a)^* \alpha = \alpha$ for all $a \in T$.
    \item $\alpha$ is $\ad(\t)$-invariant, i.e. $$\sum_{i=1}^k \alpha(Y_1,\ldots, [X, Y_i],\ldots, Y_k) = 0$$
    for all $X \in \t$ and $Y_1,\ldots, Y_k \in \g$.
\end{enumerate}

\begin{proposition}
    \label{symmetrizing to right T invariant}
    Let $G$ be a Lie group, and let $T$ be a Lie subgroup of $G$ such that $\Ad_G(T) = \{ \Ad(a): \g \rightarrow \g: a \in T\}$ is compact. Let $J$ be a left-invariant complex structure on $G$ which is also right $T$-invariant. The following are equivalent:
    \begin{enumerate}[\normalfont (i)]
        \item $(G,J)$ admits a left-invariant K\"ahler/pluriclosed/balanced metric.
        \item $(G,J)$ admits a left-invariant right $T$-invariant K\"ahler/pluriclosed/balanced metric.
    \end{enumerate}
\end{proposition}
\begin{proof}
    Let $\nu$ be a right-invariant nowhere vanishing volume form on $ \widehat T := \Ad_G(T)$. For each $k$, define 
    $$S: \Lambda^k \g^* \rightarrow \Lambda^k \g^*, \qquad (S \alpha)(X_1,\ldots, X_k) := \int_{\widehat T} (\varphi^* \alpha)(X_1,\ldots, X_k) \; \nu(\varphi). $$
    This map satisfies the following properties:
    \begin{enumerate}[(a)]
        \item $S \alpha$ is right $T$-invariant for any $\alpha \in \Lambda^k \g^*$.
        \item $S$ commutes with the exterior derivative $d$.
        \item Extending $S$ to complex forms in the obvious way, $S$ sends $(p,q)$-forms to $(p,q)$-forms.
        \item $S$ commutes with the Dolbeault operators $\partial$ and $\overline \partial$.
        \item $S$ sends positive $(p,p)$-forms to positive $(p,p)$-forms.
    \end{enumerate}
    We show (a) only; (b), (c), (d) and (e) are straightforward. Fix $\alpha \in \Lambda^k \g^*$ and $X_1,\ldots, X_k \in \g$. Define $f: \widehat T \rightarrow \R$ by $f(\varphi) = (\varphi^* \alpha) (X_1,\ldots, X_k)$. Given $\psi \in \widehat T$, we find 
    \begin{align*}
        (\psi^* S \alpha) (X_1,\ldots, X_k) &= (S \alpha) (\psi X_1,\ldots, \psi X_k) \\
        &= \int_{\widehat T}( \varphi \circ \psi)^* \alpha (X_1,\ldots, X_k) \; \nu(\varphi) \\
        &= \int_{\widehat T} (f \circ R_\psi) \; \nu =\int_{\widehat T} (R_{\psi^{-1}})^*((f \circ R_\psi) \; \nu )\\
        &=\int_{\widehat T} (f \circ R_\psi \circ R_{\psi^{-1}}) \;  (R_{\psi^{-1}})^*\nu \\
        &= \int_{\widehat T} f \; \nu = (S\alpha)(X_1,\ldots, X_k),
    \end{align*}
    where $R_\psi: \widehat T \rightarrow \widehat T$ is given by $\varphi \mapsto \varphi \circ \psi$. 

    The rest of the argument is similar to the proof of Proposition \ref{symmetrisation 1 prop}.
\end{proof}

\section{Real semisimple Lie algebras and Vogan diagrams}
\label{semisimple Lie algebras section}
\subsection{Cartan subalgebras of a real semisimple Lie algebra}
\label{Cartan subalgebras subsection}Throughout this section, let  $\g$ denote a real semisimple Lie algebra. The complexification $\g^\C$ of $\g$ is a complex semisimple Lie algebra. A\textit{ Cartan subalgebra of $\g$} is a subalgebra $\h$ such that $\h^\C$ is a Cartan subalgebra of $\g^\C$. 

Let $\h$ be a Cartan subalgebra of $\g$. Let $\Phi$ denote the roots of $(\g^\C, \h^\C)$, and let $\g^\C_\alpha$ denote the root space of $\alpha \in \Phi$. Let $\sigma: \g^\C \rightarrow \g^\C$ denote the complex conjugation map $X + \hat \i Y \mapsto X - \hat \i Y$ for $X,Y \in \g$. Since $\sigma \h^\C = \h^\C$, we obtain a root system automorphism $\sigma: \Phi \rightarrow \Phi$ given by $(\sigma \alpha)(H) = \overline{\alpha (\sigma H)}$ for all $H \in \h^\C$. Moreover, $\sigma:\g^\C \rightarrow \g^\C$ permutes the root spaces via $\sigma \g^\C_\alpha = \g^\C_{\sigma \alpha}$. We say that $\alpha \in \Phi$ is \textit{real} if $\sigma \alpha = \alpha$, \textit{imaginary} if $\sigma \alpha = - \alpha$, and \textit{complex} if $\alpha$ is neither real nor imaginary. 

Now, let $B$ denote the Killing form of $\g^\C$. Fix a root $\alpha \in \Phi$, and let $E_\alpha \in \g^\C_\alpha$ be a non-zero root vector. Then $\alpha$ is imaginary if and only if $B(E_\alpha, \sigma E_\alpha) \neq 0$. We say that $\alpha$ is \textit{compact} if $B(E_\alpha, \sigma E_\alpha) < 0$, and \textit{non-compact} if $B(E_\alpha, \sigma E_\alpha) > 0$. 

A \textit{Cartan involution} of $\g$ is a Lie algebra automorphism $\theta: \g \rightarrow \g$ such that $\theta^2 = \Id$ and $B(X, \theta X) < 0$ for all non-zero $X \in \g$. If $\theta$ is a Cartan involution, then  $B_\theta := -B(\cdot, \theta \cdot)$ is a positive-definite symmetric bilinear form on $\g$, and we have a vector space decomposition $\g =\k \oplus \p$, called the \textit{Cartan decomposition}, where $\k$ and $\p$ are the $+1$ and $-1$ eigenspaces of $\theta$, respectively. Observe that $B$ is negative-definite on $\k$ and positive-definite on $\p$. These subspaces satisfy the following relations:
\begin{equation}
    \label{Cartan decomposition relations}
    [\k , \k ]\subseteq \k, \qquad [\k, \p] \subseteq \p, \qquad [\p, \p] \subseteq \k.
\end{equation}

If $\h$ is a Cartan subalgebra of $\g$, then there exists a (not necessarily unique) Cartan involution $\theta$ of $\g$ such that $\theta \h = \h$ \cite[Proposition 6.59]{knapp}, and we have a decomposition
$$\h = \t \oplus \a, \qquad \text{where} \qquad \t = \h \cap \k, \quad \a = \h \cap \p.$$
Such a Cartan involution permutes the roots via $(\theta \alpha)(H) = \alpha(\theta^{-1} H)$ for all $H \in \h^\C$ and the root spaces via $\theta \g^\C_\alpha = \g^\C_{\theta \alpha}$.
\begin{proposition}
    \label{Cartan subalgebras facts}
    Let $\h$ be a Cartan subalgebra of $\g$, and let $\theta$ be a Cartan involution with $\theta \h = \h$.
    \begin{enumerate}[\normalfont (i)]
        \item The subspaces $\t$ and $\a$ defined above can be written as
        \begin{align}
         \label{t and a definitions}
            \begin{split}
                \t &= \{H \in \h : \alpha(H) \text{ is imaginary for all $\alpha \in \Phi$}\}, \\ \a &= \{H \in \h : \alpha(H) \text{ is real for all $\alpha \in \Phi$}\}.
            \end{split} 
        \end{align}
        In particular, $\t$ and $\a$ are independent of the choice of Cartan involution $\theta$.
        \item View $\t^*$ and $\a^*$ as subsets of $(\h^\C)^*$ in the obvious manner. Then the  real and imaginary parts of $\alpha \in \Phi$ with respect to $\sigma$ are 
        $$\Re(\alpha) = \alpha|_\a, \qquad \i\Im(\alpha) = \alpha|_{\t}.$$
        Moreover, we have $\R \Phi = \i \t^* \oplus \a^*$.
        
        \item $\theta \alpha = - \sigma \alpha$ for all $\alpha \in \Phi$.
        \item Let $\alpha \in \Phi$. Then $\alpha$ is  compact if and only if $\g^\C_\alpha \subseteq \k^\C$, and $\alpha$ is non-compact if and only if $\g^\C_\alpha \subseteq \p^\C$. 
    \end{enumerate}
    
\end{proposition}
\begin{proof}
    To show (i), fix $H \in \h$, and consider $\ad(H): \g \rightarrow \g$. If $H \in \t$, then $\ad(H)$ is skew-symmetric with respect to the inner product $B_\theta$, so $\alpha(H)$ is imaginary for all $\alpha \in \Phi$ in this case. Analogously, if $H \in \a$, then $\ad(H)$ is symmetric with respect to $B_\theta$, so $\alpha(H)$ is real for all $\alpha \in \Phi$ in this case. The opposite inclusions follow because $\C\Phi  = (\h^\C)^*.$

    The first part of (ii) follows from (i) and the formulas $2\Re(\alpha) = \alpha + \sigma \alpha $ and $2 \i \Im(\alpha) = \alpha - \sigma \alpha$. In particular, it follows that $\R \Phi \subseteq \i \t^* \oplus \a^*$. The opposite inclusion holds because $\R \Phi$ is a real vector space of dimension $\dim \h = \dim(\i \t^* \oplus \a^*)$. Part (iii) follows from (ii), because $\theta = \Id$ on $\i \t^*$ and $\theta = - \Id$ on $\a^*$.

    Finally, let us show (iv). First, by (iii), we find that $\theta \g^\C_\alpha = \g^\C_{\theta \alpha} = \g^\C_{- \sigma \alpha}$. Since $\g^\C_\alpha$ is one-dimensional, it follows that $\alpha$ is imaginary if and only if $\g^\C_\alpha \subseteq \k^\C$ or $\g^\C_\alpha \subseteq \p^\C$. Now, suppose $\g^\C_\alpha \subseteq \k^\C$. Fix a non-zero $Z = X + \hat \i Y \in \g^\C_\alpha$, where $X,Y \in \k$. Then $B(Z, \sigma Z) = B(X,X) + B(Y,Y) < 0$, so $\alpha$ is compact. Analogously, if $\g^\C_\alpha \subseteq \p^\C$, then $\alpha$ is non-compact.
\end{proof}

Thanks to Proposition \ref{Cartan subalgebras facts} Part (iv) and the relations (\ref{Cartan decomposition relations}), we immediately have the following:
\begin{corollary}
    \label{compactness addition rules}
    Let $\h$ be a Cartan subalgebra of $\g$. Let $\alpha, \beta$ be imaginary roots of $(\g^\C, \h^\C)$, and suppose $\alpha + \beta$ is a root. Then $\alpha + \beta$ is imaginary, and the compactness of $\alpha + \beta$ is determined by the following rules:
    \begin{align}
        \begin{split}
            \normalfont\text{compact} + \text{compact} &= \normalfont\text{compact}, \\
            \normalfont\text{compact} + \text{non-compact} &= \normalfont\text{non-compact}, \\
            \normalfont\text{non-compact} + \text{non-compact} &= \normalfont\text{compact}.
        \end{split}
    \end{align}
\end{corollary}

\begin{example}[Compact type]
    \label{compact example}
    We say that a real semisimple Lie algebra $\g$ is \textit{compact} if $\g$ is the Lie algebra of a compact Lie group. A real semisimple Lie algebra is compact if and only if its Killing form $B$ is negative-definite \cite[Chapter IV.4]{knapp}. 
    
    Assume that $\g$ is compact. The Cartan subalgebras of $\g$ are precisely the maximal abelian subspaces of $\g$. Up to automorphisms of $\g$, there is a unique Cartan subalgebra $\h$ of $\g$. Since $B$ is negative-definite, $- B(\cdot, \sigma \cdot) : \g^\C \times \g^\C \rightarrow \C$ is a (positive-definite) Hermitian inner product on $\g^\C$. Thus, if $\alpha$ is a root of $(\g^\C, \h^\C)$ and $E_\alpha \in \g^\C_\alpha$ is a non-zero root vector, then $B(E_\alpha, \sigma E_\alpha) < 0$, so every root of $(\g^\C, \h^\C)$ is imaginary and compact.
\end{example}

\begin{example}[Inner type]
    \label{inner example}
    Let $\g$ be a real semisimple Lie algebra, and fix a Cartan involution $\theta$ of $\g$. We say that $\g$ is \textit{inner} if $\theta$ is an inner automorphism of $\g$. For example, if $\g$ is compact, then $\g$ is also inner, because the identity map $\Id: \g \rightarrow \g$ is a Cartan involution.  A real semisimple $\g$ is inner if and only if there exists a Cartan subalgebra $\h$ of $\g$ such that $\h \subseteq \k$ \cite[Chapter VI Problems 10 and 11]{knapp}. In this case,  $\a = 0$ in the decomposition $\h = \t \oplus \a$ defined in (\ref{t and a definitions}), so all roots are imaginary.
\end{example}

\begin{example}[Complex type]
    \label{complex example}
    We say that a real semisimple Lie algebra $\g$ is \textit{complex} if $\g$ is the underlying real Lie algebra of a complex semisimple Lie algebra $\s$. From Example \ref{bi invariant complex structure 1}, recall that we have an isomorphism $\g^\C \cong \s \oplus \overline \s$, where $\overline \s$ is the complex Lie algebra obtained by replacing scalar multiplication in $\s$ with $\lambda \cdot X = \overline \lambda X$, and $\g$ embeds into $\s \oplus \overline \s$ via $X \mapsto (X,X)$. 

    The Cartan subalgebras of $\g$ are precisely the Cartan subalgebras of $\s$ viewed as real subalgebras of $\g$. Thus, $\g$ has a unique Cartan subalgebra $\h$, up to automorphisms of $\g$. Under the identification $\g^\C \cong \s \oplus \overline \s$, we may write $\h^\C = \h \oplus \overline \h$. The roots of $(\g^\C, \h^\C)$ are $\Phi = \Phi_1 \sqcup \Phi_2$, where $\Phi_1$ and $\Phi_2$ are the roots of $(\s, \h)$ and $(\overline \s, \overline \h)$, respectively.  The induced automorphism $\sigma: \Phi \rightarrow \Phi$ swaps the two factors $\Phi_1 \leftrightarrow \Phi_2$. In particular, all roots are complex.
\end{example}

\subsection{Maximally compact Cartan subalgebras}
\label{maximally compact subsection}
Let $\h$ be a Cartan subalgebra of $\g$, and let $\Phi$ denote the roots of $(\g^\C, \h^\C)$. We say that $\h$ is \textit{maximally compact} if $\sigma: \Phi \rightarrow \Phi$ has no fixed points, i.e. there are no real roots. 

\begin{remark}
    \label{uniqueness of maximally compact remark}
    Every real semisimple Lie algebra $\g$ admits a maximally compact Cartan subalgebra $\h$. If $\h'$ is another maximally compact Cartan subalgebra, then there exists an inner automorphism $f$ of $\g$ such that $\h' = f \h$ \cite[Propositions 6.59, 6.60 and 6.61]{knapp}.
\end{remark}

\begin{proposition}
    \label{prop maximally compact equivalences}
    Let $\h$ be a Cartan subalgebra of $\g$, and write $\h = \t \oplus \a$ as in (\ref{t and a definitions}). The following are equivalent:
    \begin{enumerate}[\normalfont (i)]
        \item $\h$ is maximally compact.
        \item $\dim \t$ is maximal among all Cartan subalgebras of $\g$.
        \item If $\theta$ is a Cartan involution of $\g$ with $\theta \h = \h$, then $\t$ is a maximal abelian subspace of $\k = \text{\normalfont Fix}(\theta)$.
        \item There exists a set of positive roots $\Phi^+$ such that $\sigma \Phi^+ =- \Phi^+$.
    \end{enumerate}
\end{proposition}
\begin{proof}
    The equivalence of (i), (ii) and (iii) follows from \cite[Propositions 6.60 and 6.70]{knapp}. Let us show that (i) and (iv) are equivalent. If $\h$ is not maximally compact, then there is a root $\alpha \in \Phi$ such that $\sigma \alpha = \alpha$. Thus, any set of positive roots $\Phi^+$ with $\sigma \Phi^+ = - \Phi^+$ must satisfy $\emptyset \neq (\sigma \Phi^+) \cap \Phi^+ = (- \Phi^+)\cap \Phi^+ = \emptyset$, a contradiction.

    Conversely, suppose $\h$ is maximally compact.  Let $y_1,\ldots, y_k$ be a basis for $\t^*$, and let $x_1,\ldots, x_\ell$ be a basis for $\a^*$. Since $\Phi \subseteq \i \t^* \oplus \a^*$, we can write every root as 
    $$\alpha = a_1  \i y_1 + \cdots + a_k \i y_k + a_{k+1} x_1 + \cdots + a_{k + \ell} x_\ell,$$
    where $a_i$ are real numbers. Since $x_i, y_j: \h^\C \rightarrow \C$ are real, we find that 
    $$\sigma \alpha = -a_1  \i y_1 - \cdots - a_k \i y_k + a_{k+1} x_1 + \cdots + a_{k + \ell} x_\ell.$$
    Let $i(\alpha)$ be the first index where $a_{i(\alpha)} \neq 0$. Let $\Phi^+$ be the set of roots $\alpha$ such that  $a_{i(\alpha)} (\alpha) > 0$. It is easy to see that $\Phi^+$ is a set of positive roots. Since $\h$ is maximally compact, at least one of $a_1,\ldots, a_k$ is non-zero for any root $\alpha$. Thus, $\sigma \Phi^+ = - \Phi^+$.
\end{proof}

\begin{example}
    \label{positive roots inner}
    Suppose $\g$ is inner and $\h$ is maximally compact. By Example \ref{inner example} and Proposition \ref{prop maximally compact equivalences} (ii), it follows that $\a = 0$, so $\sigma = - \Id$  on $\Phi$. Consequently, any set of positive roots $\Phi^+ \subseteq \Phi$ satisfies $\sigma \Phi^+ = - \Phi^+$.
\end{example}

\subsection{Vogan diagrams}
\label{Vogan diagrams subsection}
A \textit{Vogan diagram} is a triple $\V = (\D, \theta, \P)$, where $\D$ is a (not necessarily connected) Dynkin diagram, $\theta: \D \rightarrow \D$ is a graph automorphism of order $1$ or $2$, and $\P$ is a subset of the vertices of $\D$ fixed by $\theta$. Pictorially, we denote the two-cycles of $\theta$ with a bidirectional arrow, and paint the vertices in $\P$.

\begin{example}
    \label{complex vogan diagram 1}
    Consider $\D = \D_1 \sqcup \D_2$, where $\D_1$ and $\D_2$ are two copies of the same Dynkin diagram. Let $\theta:\D \rightarrow \D$ be the map that swaps $\D_1$ and $\D_2$. Then $\V = (\D, \theta, \emptyset)$ is a Vogan diagram.   
\end{example}

\begin{algorithm}
    \label{semisimple g to Vogan diagram}
    We now explain how Vogan diagrams arise from a real semisimple Lie algebra $\g$. Fix a maximally compact Cartan subalgebra $\h$ of $\g$, and let $\Phi$ denote the roots of $(\g^\C, \h^\C)$. Choose a set of positive roots $\Phi^+$ such that $\sigma \Phi^+ = - \Phi^+$. Such a choice always exists by Proposition \ref{prop maximally compact equivalences}. We call $(\g, \h, \Phi^+)$ a \textit{Vogan triple}. A choice of positive roots corresponds to a set of simple roots $\Delta$, which induces a Dynkin diagram $\D$. Observe that $- \sigma: \Phi \rightarrow \Phi$ is an automorphism of order $1$ or $2$  which preserves the positive roots. Thus, $-\sigma \Delta  = \Delta$,    so we obtain a graph automorphism $- \sigma: \D \rightarrow \D$ of order $1$ or $2$. Finally, let $\P$ denote the fixed vertices of $\D$ which correspond to the simple non-compact roots. Then $\V = (\D, - \sigma, \P)$ is a Vogan diagram, which we call the \textit{Vogan diagram of $(\g, \h, \Phi^+)$}. 
\end{algorithm}

\begin{example}
    \label{Vogan diagrams of types}
    Let $\g$ be a real semisimple Lie algebra, fix a maximally compact $\h$, and consider the roots of $(\g^\C, \h^\C)$.
    \begin{enumerate}[(i)]
        \item If $\g$ is compact, then all roots are compact by Example \ref{compact example}.  The only Vogan diagram of $\g$ is the Dynkin diagram of $\g^\C$ with trivial involution and no painted vertices.
        \item If $\g$ is inner, then all roots are imaginary by Example \ref{inner example}. Thus, every Vogan diagram of $\g$ has trivial involution.
        \item Let $\g$ be the underlying real Lie algebra of a complex semisimple Lie algebra $\s$. From Example \ref{complex example}, any Cartan subalgebra $\h$ of $\g$ is maximally compact, the root system of $(\g^\C, \h^\C)$ is $\Phi = \Phi_1 \sqcup \Phi_2$, where $\Phi_i$ is the root system of $\s$,  and $\sigma$ swaps $\Phi_1$ and $\Phi_2$. We have a bijection
        $$\left\{\text{Positive roots $\Phi^+_1$ of $\Phi_1$}\right\} \longleftrightarrow \left\{\text{Positive roots $\Phi^+$ of $\Phi$ with $\sigma \Phi^+ = - \Phi^+$}\right\} $$
        given by $\Phi_1^+ \mapsto \Phi^+ = \Phi_1^+ \sqcup (- \sigma \Phi_1^+)$. Any such choice determines the same Vogan diagram. The Dynkin diagram induced by $\Phi^+$ is $\D = \D_1 \sqcup \D_2$, where $\D_1$ is induced by $\Phi_1^+$ and $\D_2$ is induced by $-\sigma \Phi_1^+$. The involution $-\sigma: \D \rightarrow \D$ swaps $\D_1$ and $\D_2$. Since there are no imaginary roots, the Vogan diagram induced by $\Phi^+$ is given by Example \ref{complex vogan diagram 1}. 
    \end{enumerate}
\end{example}

\begin{proposition}{\cite[Theorems 6.74 and 6.88]{knapp}}
    \label{Vogan diagrams correspondence}
    Algorithm \ref{semisimple g to Vogan diagram} induces a bijection
     $$\left\{ \text{\normalfont Vogan triples }(\g, \h, \Phi^+)  \right\}/ \cong  \;\; \longleftrightarrow  \;\; \left\{\text{\normalfont Vogan diagrams $\V$} \right\},$$
     where we declare $(\g, \h, \Phi^+) \cong (\g', \h', (\Phi')^+) $ if there exists a Lie algebra isomorphism $f: \g \rightarrow \g'$ sending $(\h, \Phi^+)$ to $(\h', (\Phi')^+)$. In particular, every Vogan diagram arises from Algorithm \ref{semisimple g to Vogan diagram}.
\end{proposition}

We say that a Vogan diagram $\V$ is a \textit{Vogan diagram of $\g$} if $\V$ is the Vogan diagram of some Vogan triple $(\g, \h, \Phi^+)$.
\begin{corollary}
    \label{vogan diagrams corollary}
    Let $\g$ be a semisimple Lie algebra. Fix a maximally compact Cartan subalgebra $\h$ of $\g$. Let $\Phi$ denote the roots of $(\g^\C, \h^\C)$. We have a bijection 
    $$\left\{ \text{\normalfont Positive roots $\Phi^+$ with $\sigma \Phi^+ = - \Phi^+$} \right\}/ \cong  \;\; \longleftrightarrow  \;\; \left\{\text{\normalfont Vogan diagrams $\V$ of $\g$} \right\},$$
    where we declare $\Phi_1^+ \cong \Phi_2^+$ if there exists a Lie algebra automorphism $f: \g \rightarrow \g$ with $f \h = \h$ and $f \Phi^+_1 = \Phi^+_2$.
\end{corollary}

We say that a Vogan diagram $\V$ is \textit{connected} if 
\begin{enumerate}[(i)]
    \item the underlying Dynkin diagram is connected, or
    \item the underlying Dynkin diagram has two connected components, and $\theta$ sends one component to the other.
\end{enumerate}
A Vogan diagram $\V$ is connected if and only if the corresponding $\g$ is simple. If $\V$ is not connected, then its connected components correspond to the simple factors of $\g$.

\begin{proposition}
    \label{characterisation of type prop}
    Let $\g$ be a real semisimple Lie algebra. Fix a maximally compact Cartan subalgebra $\h$ of $\g$, and let $\Phi$ denote the roots of $(\g^\C, \h^\C)$. In each row of the following table, the listed conditions are equivalent:
    \begin{center}
        \textnormal{
        \begin{tabular}
        {>{\centering\arraybackslash}m{2.1cm}|
         >{\centering\arraybackslash}m{4.5cm}|
         m{7cm}}
        Type of $\g$ & Roots  & \multicolumn{1}{c}{Vogan diagrams} \\ \hline
        Compact & All roots $\Phi$ are compact.
        & $\g$ has a single Vogan diagram with trivial involution and no painted vertices. \\ \hline
        Inner & All roots $\Phi$ are imaginary.
        & All Vogan diagrams of $\g$ have trivial involution. \\ \hline
        Complex & All roots $\Phi$ are complex.
        & $\g$ has a single Vogan diagram, which is described in Example \ref{complex vogan diagram 1}.
    \end{tabular}}
    \end{center}
\end{proposition}
\begin{proof}
    Examples \ref{compact example}, \ref{inner example} and \ref{complex example} show that the first column implies the second. Let us show  that the second column implies the third; the first two rows are obvious. For the complex row, suppose that all roots are complex, and  $\V = (\D, \theta,\P)$ is any Vogan diagram. Then $\P = \emptyset$, and $\theta$ pairs up the connected components of $\D$, so $\V$ is described by Example \ref{complex vogan diagram 1}. 
    
    Finally, let us show that the third column implies the first. By Proposition \ref{Vogan diagrams correspondence}, every Vogan diagram is the Vogan diagram of a unique semisimple Lie algebra. By Example \ref{Vogan diagrams of types} (i), the Vogan diagram with $\theta = \Id$ and $\P = \emptyset$ is induced by the compact real form of the complex Lie algebra corresponding to $\D$.
    Next, any Vogan diagram with $\theta = \Id$ is the Vogan diagram of an inner semisimple Lie algebra by \cite[Chapter VI.10]{knapp} and \cite{chuah}.  Finally, Example \ref{Vogan diagrams of types} (iii) shows that the Vogan diagram described by Example \ref{complex vogan diagram 1} is induced by the complex Lie algebra corresponding to $\D_1$, viewed as a real Lie algebra.
\end{proof}

We say that a Vogan diagram is \textit{compact}/\textit{inner}/\textit{complex} if the corresponding semisimple Lie algebra $\g$ is compact/inner/complex. The simple $\g$ and connected Vogan diagrams which are neither complex nor inner are given by Table \ref{mixed diagrams}.

\begin{table}[h]
\centering
\begin{tabular}{c|c}
    Family & Vogan diagram \\ \hline
    $\sl(n+1,\R)$, $n$ even & \dynkin[edge length=1cm, root radius=.1cm, involutions={16;25;34} ]{A}{oo.oo.oo} \\ 
    $\sl(n+1,\R)$, $n$ odd & \dynkin[edge length=1cm, root radius=.1cm, involutions={15;24} ]{A}{o.o*o.o} \\ 
    $\sl((n+1)/2,\mathbb H)$, $n$ odd & \dynkin[edge length=1cm, root radius=.1cm, involutions={15;24} ]{A}{o.ooo.o} \\ 
    $\mathfrak{so}(p,q)$, both $p$ and $q$ odd  & \dynkin[edge length=1cm, root radius=.1cm, involutions={*[relative, out=0, in=180]45}]{D}{tt.too} \\
    $E\text{I}$ or $E\text{IV}$ &  \dynkin[edge length=1cm, root radius=.1cm, involutions={16;35}]{E}{ototoo}
\end{tabular}
\caption{Connected Vogan diagrams which are neither complex nor inner. A vertex \dynkin[edge length=1cm, root radius=.1cm]{A}{t} denotes either \dynkin[edge length=1cm, root radius=.1cm]{A}{*} or \dynkin[edge length=1cm, root radius=.1cm]{A}{o}.}
\label{mixed diagrams}
\end{table}

\begin{remark}
    \label{cartan involution inner}
    When $\g$ is inner and $\h$ is maximally compact, there is a unique Cartan involution $\theta$ such that $\theta \h = \h$. Indeed, if $\theta$ is such a Cartan involution, then the corresponding Cartan decomposition must be given by 
    $$\k^\C = \h^\C \oplus \bigoplus_{\alpha \in \Phi_{\text{c}}} \g^\C_\alpha, \qquad \p^\C = \bigoplus_{\alpha \in \Phi_{\text{nc}}} \g^\C_\alpha,$$
    where $\Phi_{\text{c}}$ and $\Phi_{\text{nc}}$ denote the compact and non-compact roots, respectively.
\end{remark}

\subsection{Characterisations of Vogan diagrams in Table \ref{vogan table}} 
\label{table 1 subsection}Let $\V$ be a Vogan diagram. By Proposition \ref{Vogan diagrams correspondence}, $\V$ determines a root system $\Phi$, a partition of $\Phi$ into complex, compact and non-compact roots, and a choice of positive roots $\Phi^+$. Let $\alpha_1,\ldots, \alpha_r$ be the simple roots determined by $\Phi^+$, which we identify with the vertices of $\V$. Any root $\alpha \in \Phi$ can be written as 
$$\alpha = \sum_{i=1}^r c_i(\alpha) \alpha_i,$$
where $c_i(\alpha)$ are integers.

If the underlying Dynkin diagram of $\V$ is connected (i.e. $\V$ is connected and not complex), then there is a unique root $\widetilde \alpha \in \Phi$ such that $\sum_{i=1}^r c_i(\widetilde \alpha)$ is maximal. We call $\widetilde \alpha$ the \textit{highest root}.

\begin{lemma}
    \label{single painted vertex lemma}
    Suppose $\V$ is a connected inner Vogan diagram with a single painted vertex, $\alpha_k$. The following are equivalent:
    \begin{enumerate}[\normalfont(i)]
        \item $\V$ is listed in Table \ref{vogan table}.
        \item The coefficient of $\alpha_k$ in the highest root $\widetilde \alpha$ is $1$, i.e. $c_k(\widetilde \alpha) = 1$.
        \item The coefficient of $\alpha_k$ in any positive root is at most $1$.
        \item $\alpha_k + \alpha$ is not a root for any non-compact positive root $\alpha$.
    \end{enumerate}
\end{lemma}
\begin{proof}
    \cite[pages 58 and 66]{humphreys} immediately imply that (i) and (ii) are equivalent. By \cite[Theorem 5.5 (d)]{knapp} or \cite[Theorem 10.1]{hall}, the highest root $\widetilde \alpha$ has the following property: if $\alpha$ is another positive root, then $c_i(\alpha) \leq c_i(\widetilde \alpha)$ for $i = 1,\ldots, r$. Therefore,  (ii) implies (iii). 

    Next, suppose (iii) holds, and let $\alpha \in \Phi^+$ be non-compact. Since $\alpha_k$ is the only non-compact simple root, we must have $c_k(\alpha) = 1$. Thus, $c_k(\alpha_k + \alpha) = 2$, so $\alpha_k + \alpha$ cannot be a root. Thus, (iii) implies (iv).

    Finally, let us show that (iv) implies (ii) by showing the contrapositive. Suppose $c_k(\widetilde \alpha) \geq 2$. By \cite[Chapter II Problem 7]{knapp}, we may write $\widetilde \alpha = \alpha_{i_1} + \alpha_{i_2} + \cdots + \alpha_{i_m},$
    where each partial sum from the left is a positive root. Let $\ell \geq 2$ be the second index such that $i_\ell = k$. Then $\alpha := \alpha_{i_1} + \alpha_{i_2} + \cdots + \alpha_{i_{\ell -1}}$ is a non-compact positive root, and $\alpha_k + \alpha$ is a root.
\end{proof}

Let $\D$ denote the underlying Dynkin diagram of $\V$.
A \textit{path} in $\D$ is a finite sequence of distinct vertices $\alpha_{i_1}, \alpha_{i_2},\ldots, \alpha_{i_\ell}$ such that $\alpha_{i_t}$ and $\alpha_{i_{t + 1}}$ share an edge for all $1 \leq t \leq \ell -1$. 
\begin{lemma}
    \label{root chain lemma}
    Suppose $\D$ is connected and  $\alpha_{i_1}, \alpha_{i_2},\ldots, \alpha_{i_\ell}$ is a path. Then $\alpha_{i_j} + \alpha_{i_{j+1}} + \cdots + \alpha_{i_k}$ is a root for any $1 \leq j \leq k \leq \ell$. 
\end{lemma}
\begin{proof}
    Fix $1 \leq j \leq k \leq \ell$, and 
    let $\langle \cdot, \cdot \rangle$ denote the inner product of $\Phi$. 
    Then $\langle \alpha_{i_j}, \alpha_{i_{j+1}} \rangle < 0$, so $\alpha_{i_{j}} + \alpha_{i_{j+1}}$ is a root by \cite[Proposition 2.48 (e)]{knapp}. Next, we find  $\langle \alpha_{i_{j}} + \alpha_{i_{j+1}}, \alpha_{i_{j+2}}\rangle  = \langle \alpha_{i_{j+1}}, \alpha_{i_{j+2}}\rangle   < 0$, so $\alpha_{i_{j}} + \alpha_{i_{j+1}} + \alpha_{i_{j+2}}$ is a root. The result follows by continuing in this manner.
\end{proof}

\begin{proposition}
    \label{table 1 characterisation prop}
    A connected Vogan diagram $\V$ is listed in Table \ref{vogan table} if and only if $\V$ is inner and satisfies the following property:
    \begin{equation}
        \label{table 1 property}
        \text{$\alpha + \beta$ is not a root for any non-compact positive roots $\alpha,\beta$.}
    \end{equation}
\end{proposition}
\begin{proof}
    Suppose $\V$ is listed in  Table \ref{vogan table}. Clearly, $\V$ is inner and has either zero or one painted vertices. If there are no painted vertices (i.e. $\V$ is compact), then there are no non-compact roots, so (\ref{table 1 property}) holds vacuously. Suppose $\V$ has one painted vertex, $\alpha_k$. If $\alpha, \beta \in \Phi^+$ are non-compact, then $c_k(\alpha) = c_k(\beta) = 1$, by Lemma \ref{single painted vertex lemma} (iii). Therefore, $c_k( \alpha + \beta) = 2$, so $\alpha + \beta$ cannot be a root by Lemma \ref{single painted vertex lemma} (iii). 

    Conversely, suppose $\V$ is not listed in Table \ref{vogan table}. The diagrams in  Table \ref{vogan table} are all inner, so it remains to consider when $\V$ is inner. Suppose $\V$ has exactly one painted vertex, $\alpha_k$. Then Lemma \ref{single painted vertex lemma} (iv) implies that there exists non-compact $\alpha \in \Phi^+$ such that $\alpha_k + \alpha$ is a root, so (\ref{table 1 property}) does not hold. Finally, suppose $\V$ has at least two painted vertices. Choose a path $\alpha_{i_1}, \alpha_{i_2},\ldots, \alpha_{i_\ell}$ in the Dynkin diagram such that $\alpha_{i_1}$ and $\alpha_{i_\ell}$ are non-compact, while $\alpha_{i_2},\ldots, \alpha_{i_{\ell -1}}$ are compact.
    By Corollary \ref{compactness addition rules} and Lemma \ref{root chain lemma},  $\alpha = \alpha_{i_1} + \cdots + \alpha_{i_{\ell - 1}}$ and $\beta = \alpha_{i_\ell}$ are non-compact positive roots, and $\alpha+ \beta$ is a root.
\end{proof}

\subsection{$\sigma$-adapted root vectors} 
\label{root vectors subsection}Fix a maximally compact Cartan subalgebra $\h$ of $\g$, and let $\Phi$ be the roots of $(\g^\C, \h^\C)$. Let $B$ be the Killing form of $\g$. For each root $\alpha$, let $H_\alpha \in \h^\C$ denote the $(B|_{\h^\C \times \h^\C
    })$-dual of $\alpha$.
\begin{lemma}
    \label{sigma adapted root vectors lemma}
    There exist root vectors $E_\alpha \in \g^\C_\alpha$ such that 
    $$B(E_\alpha, E_{-\alpha}) = 1, \qquad [E_\alpha, E_{-\alpha}] = H_\alpha, \qquad \sigma E_\alpha = \begin{cases}
        E_{\sigma \alpha} & \text{if $\alpha$ is complex} \\
        -s_\alpha E_{-\alpha} & \text{if $\alpha$ is imaginary},
    \end{cases}$$
    where $s_\alpha = 1$ if $\alpha$ is compact and $s_\alpha = -1$ if $\alpha$ is non-compact.
\end{lemma}
\begin{proof}
    Fix a pair of compact roots $ \pm \alpha \in \Phi$. Choose a root vector $E_\alpha \in \g^\C_\alpha$. Since $B(E_\alpha, \sigma E_{\alpha}) < 0$, we may assume that $B(E_\alpha, \sigma E_{\alpha}) = -1$ after rescaling $E_\alpha$. Setting $E_{-\alpha} := - \sigma E_\alpha$ gives us $B(E_\alpha, E_{-\alpha}) = 1$ and $[E_\alpha, E_{-\alpha}] =H_\alpha$. The case when $\pm\alpha$ are non-compact is analogous.

    Next, fix a quadruple of complex roots $\pm \alpha, \pm \sigma \alpha \in \Phi$. Choose any root vector $E_\alpha \in \g^\C_\alpha$. Choose $E_{-\alpha} \in \g^\C_{-\alpha}$ so that $B(E_\alpha, E_{-\alpha}) = 1$, and set $E_{\sigma \alpha} := \sigma E_{\alpha}$, $E_{-\sigma \alpha} := \sigma E_{-\alpha}$. We are done if we show that $B(E_{\sigma \alpha}, E_{-\sigma \alpha }) = 1$. Indeed, we find 
    $B(E_{\sigma \alpha}, E_{-\sigma \alpha })  = B(\sigma E_{ \alpha}, \sigma E_{- \alpha }) = \overline{B(E_\alpha, E_{-\alpha})} = 1,$
    where the second equality holds because $B$ is real.
\end{proof}

\subsection{Left-invariant right $T$-invariant metrics} 
\label{right T invariant subsection}
Let $G$ be a real semisimple Lie group with Lie algebra $\g$. Let $\h$ be a maximally compact Cartan subalgebra of $\g$,  and write $\h =  \t \oplus \a$ as in (\ref{t and a definitions}). Let $T$ be the connected subgroup of $G$ with Lie algebra $\t$.

\begin{lemma}  
    \label{AdGT is compact}
    The Lie group $\Ad_G(T)$ is a maximal torus of a  maximal compact subgroup of $\Ad(G)$. In particular, $\Ad_G(T)$ is compact.
\end{lemma}
\begin{proof}
    Recall that there exists a Cartan involution $\theta: \g \rightarrow \g$ such that $\theta  \h = \h$, and we have $\t = \h \cap \k$, where $\k$ is the fixed-point subspace of $\theta$. Since $\h$ is maximally compact, this implies that $\t$ is a maximal abelian subspace of $\k$ by Proposition \ref{prop maximally compact equivalences}.   Let $\widehat K$ denote the Lie subgroup of $\Ad(G)$ whose Lie algebra is $\k$. The subgroup $\widehat K$ is a maximal compact subgroup of $\Ad(G)$ \cite[Chapter VI.3]{knapp}. Since $\Ad_G(T)$ is the connected subgroup of $\widehat K$ with Lie algebra $\t$, it follows that  $\Ad_G(T) \subseteq \widehat K$ is a maximal torus \cite[Proposition 4.30]{knapp}.
\end{proof}

\begin{proposition}
    \label{right T invariant metrics form}
    A left-invariant Riemannian  metric $g$ on $G$ is right $T$-invariant if and only if 
    \begin{enumerate}[\normalfont (i)]
        \item $g(\h^\C, \g^\C_\alpha) = 0$ for each $\alpha \in \Phi$, and 
        \item  $g(\g^\C_\alpha, \g^\C_\beta) \neq 0$ implies  $\beta \in \{-\alpha, \sigma \alpha\}$ for any $\alpha, \beta \in \Phi$.
    \end{enumerate}
\end{proposition}

\begin{proof}
    Suppose $g$ is right $T$-invariant. This is equivalent to $g$ being $\ad_\g(\t)$-invariant, i.e. $g(\ad(H) \cdot,\cdot) + g(\cdot, \ad(H) \cdot) = 0$ for all $H \in \t$ (see \S \ref{symmeterisation 2 subsection}). To show (i), fix $H' \in \h^\C$ and $X \in \g^\C_\alpha$ for some root $\alpha \in \Phi$. Since $\h$ is maximally compact, there exists $H \in \t$ such that $\alpha(H) \neq 0$ (Proposition \ref{Cartan subalgebras facts}). We find
    $$g(H',X) = \frac{1}{ \alpha(H)}g(H',[H,X]) = \frac{1}{ \alpha(H)} g([H',H],X) = 0. $$
    To show (ii), fix $X \in \g^\C_{\alpha}$ and $Y \in \g^\C_\beta$ for roots $\alpha, \beta \in \Phi$, and suppose that $g(X,Y) \neq 0$. By \cite[page 5]{vogan_1979}, the imaginary part of a root determines the real part of a root, up to a sign. Thus, $\beta \in \{-\alpha, \sigma \alpha\}$ if and only if  $\i \Im(\alpha + \beta) = \alpha|_\t + \beta|_\t = 0$. Fix $H \in \t$, and let us show that $\alpha(H) + \beta(H) = 0$. If $\alpha(H) = \beta(H) = 0$, then we are done. Without loss of generality, assume that $\alpha(H) \neq 0$. Then 
    $$g(X,Y) = \frac{1}{\alpha(H)} g([H,X],Y) =-\frac{1}{\alpha(H)} g(X,[H,Y])  = - \frac{\beta(H)}{\alpha(H)} g(X,Y),$$
    so rearranging gives $\alpha(H) + \beta(H) = 0$.

    Conversely, suppose (i) and (ii) hold. Fix $H \in \t$, $H' \in \h^\C$, $X \in \g^\C_\alpha$, and $Y \in \g^\C_\beta$ for roots $\alpha , \beta \in \Phi$. Then 
    \begin{align*}
        g([H,H'],X) + g(H', [H,X]) = 0 + \alpha(H) g(H',X) &= 0, \\
        g([H,X],Y) + g(X, [H,Y]) =(\alpha + \beta)(H) g(X,Y) &= 0,
    \end{align*}
    as desired.
\end{proof}

\section{Regular complex structures on real semisimple Lie groups}
\label{regular complex structures section}
\subsection{Classification of regular complex structures}
\label{regular classification subsection}
 Throughout this section, let $G$ be an even-dimensional  real semisimple Lie group with Lie algebra $\g$, and let $\h$ denote a Cartan subalgebra of $\g$. Let $\Phi$ denote the roots of $(\g^\C, \h^\C)$, and let $\g^\C_\alpha$ denote the root space of $\alpha \in \Phi$. Let $H_\alpha \in \h^\C$ denote the dual of $\alpha \in \Phi$ with respect to the Killing form. 

We say that a left-invariant complex structure $J: \g \rightarrow \g$ on $G$ is \textit{$\h$-regular} if 
$$\ad(H) \circ J = J \circ \ad(H) \qquad \text{ for all $H \in \h$}.$$

\begin{remark}
    \label{equivalent definitions of regular}
    For a left-invariant complex structure $J$ on $G$, the following are equivalent:
    \begin{enumerate}[\normalfont (i)]
        \item $J$ is $\h$-regular.
        \item $[\h^\C, \q ] \subseteq \q$, where $\q$ is the $\i$-eigenspace of $J:\g^\C \rightarrow \g^\C$.
        \item $J$ is right $G_\h$-invariant, where $G_\h$ is the connected subgroup of $G$ with Lie algebra $\h$.
    \end{enumerate}
\end{remark}

We say that a left-invariant complex structure $J$ on $G$ is \textit{regular} if it is $\h$-regular with respect to some Cartan subalgebra $\h$ of $\g$.

\begin{remark}
    \label{closed subsets prop}
    A subset of roots $\RR \subseteq \Phi$ is called \textit{closed} if it has the following property: for any $\alpha, \beta \in \RR$, if $\alpha + \beta$ is a root, then $\alpha + \beta \in \RR$.  
    Suppose $\RR$ is closed, and set 
    $\RR_0 := \RR \cap (-\RR)$ and $ \RR_1 := \RR \backslash \RR_0.$
    Then 
    \begin{enumerate}[\normalfont (i)]
        \item $\RR_0$ is a closed root subsystem of $\Phi$.
        \item If $\alpha \in \RR_1$, $\beta \in \RR$, and $\alpha + \beta$ is a root, then $\alpha  + \beta \in \RR_1$.
    \end{enumerate}
\end{remark}

 A (complex) subalgebra $\f$ of $\g^\C$ is called \textit{$\h^\C$-regular} if 
$$[\h^\C, \f] \subseteq \f.$$
The following proposition characterises all $\h^\C$-regular subalgebras of $\g^\C$:

\begin{proposition}[{\cite[Chapter 6, Proposition 1.1]{vinberg}}]
    \label{regular subalgebras structure}
    Let $\RR \subseteq \Phi$ be a closed subset of roots. Let $\l \leq \h^\C$ be a subspace such that $H_\alpha \in \l$ for all $\alpha \in \RR \cap (-\RR)$. Then 
    $$\f(\RR, \l) := \l \oplus \bigoplus_{\alpha \in \RR}\g^\C_\alpha$$
    is a $\h^\C$-regular subalgebra of $\g^\C$.
    Conversely, if $\f$ is a $\h^\C$-regular subalgebra of $\g^\C$, then there exists a unique closed subset of roots $\RR \subseteq \Phi$ and a subspace $\l \subseteq \h^\C$ containing $H_\alpha$ for $\alpha \in \RR \cap (-\RR)$ such that $\f = \f(\RR, \l)$.
\end{proposition}

\begin{remark}
    \label{construction remark}
    Observe that a subspace $\q \leq \g^\C$ is the $\i$-eigenspace of an $\h$-regular complex structure $J: \g \rightarrow \g$ on $G$ if and only if $\q$ is an $\h^\C$-regular subalgebra of $\g^\C$ satisfying $\q \oplus \sigma \q = \g^\C$.
    Now, let  $\q$ denote a $\h^\C$-regular subalgebra $$\q = \f(\RR, \l) = \l \oplus \bigoplus_{\alpha \in \RR} \g^\C_\alpha,$$ where $\RR$ is a closed subset of roots and $\l \leq \h^\C$ contains $H_\alpha$ for all $\alpha \in \RR \cap (- \RR)$. Then $\q \oplus \sigma \q = \g^\C$ if and only if  
    $$\RR \sqcup \sigma \RR = \Phi \qquad \text{and} \qquad \l \oplus \sigma \l = \h^\C.$$
\end{remark}
We say that two subdiagrams of a Dynkin diagram are \textit{disconnected} if they are disjoint and no vertex of one subdiagram shares an edge with a vertex in the other subdiagram.

\begin{algorithm}
    \label{regular complex structures construction algorithm}
    \begin{enumerate}[(i)]
    \item Choose a maximally compact Cartan subalgebra $\h$ of $\g$.
        \item Choose a set of positive roots $\Phi^+ \subseteq \Phi$ such that $\sigma \Phi^+ = - \Phi^+$. Let $\Delta$ denote the simple roots determined by $\Phi^+$. Recall that $-\sigma$ permutes the simple roots.
        \item Choose a (possibly empty) set of simple roots $\Delta_0\subseteq \Delta$ such that the subdiagrams of  $\Delta_0$ and $- \sigma \Delta_0$ are disconnected in the Dynkin diagram of $\Delta$.
        \item Choose a complex subspace $\l \leq \h^\C$ such that  $H_\alpha \in \l$ for each $\alpha \in \Delta_0$, and  $\l \oplus \sigma \l = \h^\C$.
    \end{enumerate}
\end{algorithm}
\begin{proposition}
    \label{regular structure prop}
Let $G$ be a real semisimple Lie group, and let $\h$ be a Cartan subalgebra of $\g$.
\begin{enumerate}[\normalfont (A)]
    \item Suppose $\h$ is maximally compact.  Let $(\Phi^+, \Delta_0, \l)$ be chosen as in Algorithm \ref{regular complex structures construction algorithm}. Set $\RR_0 := \C \Delta_0 \cap \Phi$, $\RR_0^+ := \RR_0 \cap \Phi^+$, and $\RR = (-\RR_0^+) \sqcup \Phi^+ \backslash( - \sigma \RR_0^+).$ Then 
    $$\q := \f(\RR,\l) =    \l \oplus \bigoplus_{\alpha \in \RR}\g^\C_\alpha$$
    is an $\h^\C$-regular subalgebra of $\g^\C$ satisfying $\q \oplus \sigma \q = \g^\C$, and the corresponding left-invariant complex structure $J = J (\Phi^+, \Delta_0, \l) $ is $\h$-regular.
    \item Conversely, suppose $J$ is an $\h$-regular complex structure on $G$. Then $\h$ is maximally compact, and $J = J(\Phi^+, \Delta_0, \l)$ for some choice of $(\Phi^+, \Delta_0, \l)$ described in Algorithm \ref{regular complex structures construction algorithm}.
\end{enumerate}
\end{proposition}

We give a proof of Proposition \ref{regular structure prop} in \S \ref{regular structure proof}.

\begin{remark}
    \label{choices remark}
   We give a few remarks on the choices made in  Algorithm \ref{regular complex structures construction algorithm}. In particular, we explain why any even-dimensional real semisimple $G$ admits a regular complex structure.
   \begin{enumerate}[(i)]
    \item By Remark \ref{uniqueness of maximally compact remark}, a maximally compact $\h$ can always be chosen, and every regular complex structure on $G$ is equivalent to one which is $\h$-regular with respect to  a fixed maximally compact Cartan subalgebra $\h$.
    \item By Proposition \ref{prop maximally compact equivalences}, a $\Phi^+$ as in Algorithm \ref{regular complex structures construction algorithm} can always be chosen. By Corollary \ref{vogan diagrams corollary}, the choice of $\Phi^+$ in Algorithm \ref{regular complex structures construction algorithm} is equivalent to a choice of Vogan diagram for $\g$, up to equivalence of left-invariant complex structures on $G$. In particular, if $\g$ has a unique Vogan diagram, then the choice of $\Phi^+$ does not matter (e.g. if $\g$ is compact or complex, see Proposition \ref{characterisation of type prop}).
    \item The roots $\Delta_0$ chosen as in  Algorithm \ref{regular complex structures construction algorithm} must be complex, i.e. not fixed points of $-\sigma:\Phi \rightarrow \Phi$.
    \item Suppose $\h$, $\Phi^+$ and $\Delta_0$ are chosen as in  Algorithm \ref{regular complex structures construction algorithm}. 
    Assume that $\g$ is even-dimensional. Then $\dim_\C \g^\C = \dim_\C \h^\C + 2|\Phi^+| $ is even, so $\dim_\C \h^\C = 2m$ for some integer $m$. By identifying $\h^\C \cong \C^{2m}$ via an isomorphism $\h \cong \R^{2m}$, 
    Lemma \ref{grassmannian lemma} implies that we can always choose $\l$ as in Algorithm \ref{regular complex structures construction algorithm}, and the set of all such $\l$ forms a smooth manifold $\M$ of real dimension 
    $$\dim \mathcal M = \dim (\h )\left( \frac12 \dim( \h) - |\Delta_0|\right).$$
\end{enumerate}
\end{remark}

\begin{example}
    If $\g$ is inner, then any set of positive roots $\Phi^+$ satisfies $\sigma \Phi^+ = - \Phi^+$ (Example \ref{positive roots inner}), and the only choice of $\Delta_0$ is $\Delta_0 = \emptyset$, because all roots are imaginary (Proposition \ref{characterisation of type prop}). Thus, the $\h$-regular complex structures on $G$ are precisely 
    $$\q =\l \oplus \bigoplus_{\alpha \in \Phi^+} \g^\C_\alpha,$$
    where $\Phi^+$ is any set of positive roots, and $\l \leq \h^\C$ is any subspace with $\l \oplus \sigma \l  = \h^\C$.
\end{example}

\begin{example}
    Suppose $\g$ is complex, i.e. $\g$ is the underlying real Lie algebra of a complex semisimple Lie algebra $\s$. From Example \ref{bi invariant complex structure 1}, we have a bi-invariant complex structure $J: \g \rightarrow \g$ given by $J X := \i X$, where $\i X$ is scalar multiplication in $\s$.
    By Proposition \ref{characterisation of type prop},  $\g$ has a unique Vogan diagram: the underlying Dynkin diagram is $\D_1 \sqcup \D_2$, where $\D_i$ is the Dynkin diagram of $\s$. The involution $-\sigma$ swaps the copies $\D_1$ and $\D_2$ (see Proposition \ref{characterisation of type prop}).  
    Choose $\Delta_0$ to be the simple roots corresponding to the first copy $\D_1$. Then $-\sigma\Delta_0$ corresponds to $\D_2$, so $\Delta_0$ and $-\sigma\Delta_0$ are disconnected. In this case, the only choice of $\l$ is $\l = \C\{H_\alpha : \alpha \in \D_1\}$, and the corresponding $\h$-regular complex structure is the bi-invariant complex structure $J$ on $G$.
\end{example}

\begin{remark}
    The author believes that there are a couple of minor errors in \cite{snow_1986}:
    \begin{enumerate}[(i)]
        \item For an $\h$-regular complex structure $\q = \f(\RR,\l) = \l \oplus \bigoplus_{\alpha \in \RR}\g^\C_\alpha$, \cite[Theorem in \S2.1]{snow_1986} claims that we can always write $\l = \h_0 \oplus \h_1$, where $\h_0 = \C\{H_\alpha: \alpha \in \RR \cap (-\RR)\}$, and $\h_1$ is a subspace of the $B$-orthogonal complement of $\h_0 \oplus \sigma \h_0$.  Example \ref{counterexample for snow} gives a counterexample. 
        \item \cite[\S2.2, page 202]{snow_1986} claims that for any choice of positive roots $\Phi_1^+$ and $\Phi_2^+$ such that $\sigma \Phi^+_i = - \Phi^+_i$, there is an automorphism of $\g$ preserving $\h$ which sends $\Phi_1^+$ to $\Phi_2^+$. This is false: if such an automorphism exists, then the Vogan diagrams of $\Phi_1^+$ and $\Phi_2^+$ are the same, but $\g = \su(2,1)$ has two distinct Vogan diagrams (see Table \ref{rank 2 table}).
    \end{enumerate}
\end{remark}

\begin{example}
    \label{counterexample for snow}
    The Lie algebra $\g= \sl(5,\R)$  has only one Vogan diagram, which is
    $$\dynkin[
  labels={\alpha_1,\alpha_2,\alpha_3,\alpha_4},
  label directions={above,above,above,above},
  edge length=1cm,
  root radius=.1cm,
  involutions={14;23}
]{A}{oooo}.$$
 If $H_i$ denotes the $B$-dual of $\alpha_i$, then $\sigma H_1 = - H_4$ and $\sigma H_2 = - H_3$. Let $\RR = \{-\alpha_1\} \sqcup \Phi^+ \backslash\{\alpha_4\}$, so that $\h_0 = \C H_1$. Observe that $\RR$ is a closed subset of $\Phi$ satisfying $\RR \sqcup \sigma \RR = \Phi$. Let $\l := \C \{H_1, H_3\}$. Then $\l$ contains $\h_0$, and $\l \oplus \sigma \l  = \h^\C$. By Remark \ref{construction remark}, $$\q =  \l \oplus \bigoplus_{\alpha \in \RR}\g^\C_\alpha= \C \{H_1, H_3\} \oplus \g^\C_{-\alpha_1}\oplus \bigoplus_{\alpha \in \Phi^+ \backslash\{\alpha_4\}}\g^\C_\alpha $$ defines  a $\h$-regular complex structure. The orthogonal complement of $\h_0 \oplus \sigma \h_0 = \C\{H_1,H_4\}$ is $\s = \C\{H_1 + 2 H_2, H_4 + 2 H_3\}$, but clearly $\l \cap \s = 0$.
\end{example}

\begin{lemma}
    \label{grassmannian lemma}
    Fix $m \geq 1$, and let $\sigma$ denote the standard complex conjugation map of $\C^{2m}$. Let $\Gr_m(\C^{2m})$ denote the Grassmannian of $m$-dimensional complex subspaces of $\C^{2m}$. Fix an $r$-dimensional complex subspace $U \subseteq \C^{2m}$ such that $U \cap \sigma U = 0$. Then 
    $$\M = \left\{ V \in \Gr_m(\C^{2m}) : U \subseteq V \text{ and } \V \oplus \sigma V = \C^{2m}\right\}$$
    is an embedded submanifold of $\Gr_m(\C^{2m})$ with $\dim_\R \M = 2m (m-r).$
\end{lemma}
\begin{proof}
    Let $e_1,\ldots, e_{2m}$ denote the standard basis of $\C^{2m}$. For each $m$-element subset $I$ of  $\{1,\ldots,2m\}$, let $V_I = \spann_\C\{e_i : i \in I\}$, $I^c = \{1,\ldots, 2m\} \backslash I$, 
    $$\U_I := \left\{ V \in \Gr_m(\C^{2m}) : V \cap V_{I^c} = 0\right\},$$
    $$F_I: \Hom_\C(V_I, V_{I^c}) \rightarrow \U_I, \qquad F_I(X) := \{v + Xv : v \in V_I\}.$$
    Each $\U_I$ is an open subset of $\Gr_m(\C^{2m})$, each $F_I:\Hom_\C(V_I, V_{I^c}) \rightarrow \U_I$ is a diffeomorphism, and the sets $\U_I$ cover $\Gr_m(\C^{2m})$.

    Now, define 
    $$\CC := \left\{ V \in \Gr_m(\C^{2m}) :V \oplus \sigma V = \C^{2m}\right\}.$$
    Let us show that $\CC$ is an open subset of $\Gr_m(\C^{2m})$. It suffices to show that $F_I^{-1}(\CC)$ is open for all $I$. Choose an identification $\Hom_\C(V_I, V_{I^c}) \cong M_m(\C)$ via the basis $e_1,\ldots, e_{2m}$. A straightforward computation shows that 
    $$F_I^{-1}(\CC) = \left\{ X \in M_m(\C) : \det( X - \overline X) \neq 0\right\},$$
    which is open.

    Next, let 
    $$\N :=  \left\{ V \in \Gr_m(\C^{2m}) : U \subseteq V\right\}.$$
    Let us show that $\N$ is an embedded submanifold of $\Gr_m(\C^{2m})$ of dimension $2m (m-r)$. It suffices to show that $F_I^{-1}(\N)$ is either empty or an embedded submanifold of $\Hom_\C(V_I, V_{I^c})$ of dimension $2m (m-r)$ for each $I$. Fix $I$, and suppose $F_I^{-1}(\N)$ is not empty, so that there exists $V \in   \Gr_m(\C^{2m})$ with $U \subseteq \V$ and $V \cap \V_{I^c} = 0$. In particular, $U \cap V_{I^c} = 0$, so $\pi_I |_U : U \rightarrow V_I$ is injective, where $\pi_I:\C^{2m} \rightarrow V_I$ is the projection map with respect to $\C^{2m} = V_I \oplus V_{I^c}$. Set $U_0 := \pi_I(U)$, and let $L:U_0 \rightarrow V_{I^c}$ be the composition 
    $$L: U_0 \xrightarrow{(\pi_I|_U)^{-1}} U \xhookrightarrow{} \C^{2m} \xrightarrow {\pi_{I^c}} V_{I^c}.$$
    A straightforward computation shows that 
    $$F^{-1}_I(\N) = \left\{ X \in \Hom_\C(V_I, V_{I^c}) : X|_{U_0} = L\right\},$$
    which is an embedded submanifold of $\Hom_\C(V_I, V_{I^c})$ of dimension $2m (m-r)$.

    We are done if we show that $\M = \CC \cap \N$ is non-empty. Let $E$ be the orthogonal complement of $U \oplus \sigma U$ with respect to the standard Hermitian inner product of $\C^{2m}$. Then $\sigma E = E$ and $E$ is even-dimensional, so there exists a complex subspace $W \subseteq E$ with $W \oplus \sigma W  = E$. Then $U \oplus W \in \M$.
\end{proof}

\subsection{Proof of Proposition \ref{regular structure prop}}
\label{regular structure proof}
Let us begin by showing (A). In the notation of Proposition \ref{regular structure prop}, it suffices to show the following (see Remark \ref{construction remark}):
    \begin{enumerate}
         \item[(A1)] $\RR$ is a closed subset of $\Phi$.
         \item[(A2)] $\RR \sqcup \sigma \RR = \Phi$.
         \item[(A3)] $H_\alpha \in \l$ for all $\alpha  \in \RR \cap (- \RR)$.
    \end{enumerate}

    Let us first show that $\RR$ is a closed subset of $\Phi$.
    For each root $\alpha \in \Phi$, let us write $\alpha$ in terms of the simple roots, i.e.
    $$\alpha = \sum_{\varphi \in \Delta} c_\varphi(\alpha) \varphi$$
    for some integers $c_\varphi(\alpha)$. Recall that the $c_\varphi(\alpha)$ are either all non-negative or  all non-positive, and $\alpha$ is a positive root if and only if $c_\varphi(\alpha) \geq 0$ for all $\varphi \in \Delta$. The \textit{support of $\alpha$} is
    $$\supp(\alpha) := \left\{ \varphi \in \Delta: c_\varphi(\alpha) \neq 0 \right\}.$$
    The subdiagram of the Dynkin diagram of $\Delta$ corresponding to  $\supp(\alpha) \subseteq \Delta$ is always connected \cite[page 179, Corollary 3]{bourbaki}. Observe that $\alpha \in \RR_0$ if and only if $\supp(\alpha) \subseteq \Delta_0$. Define $\RR_1 := \RR \backslash \RR_0 = \Phi^+ \backslash(\RR_0^+ \sqcup - \sigma \RR_0^+)$, and 
    $\Delta_1 := \Delta \backslash(\Delta_0 \sqcup - \sigma \Delta_0)$. Note that $\Delta_1$ does not necessarily generate $\RR_1$.  Let us show the following statement:
    \begin{equation}
        \label{RR1 condition}
        \text{ Let $\alpha \in \Phi^+$ be a positive root. Then $\alpha \in \RR_1$ if and only if $\supp(\alpha) \cap \Delta_1 \neq \emptyset$.}
    \end{equation}
    First, suppose that $\alpha \in \RR_1$. For the sake of contradiction, suppose that $\supp(\alpha) \cap \Delta_1 = \emptyset$. Then $\supp(\alpha) \subseteq \Delta_0 \sqcup - \sigma \Delta_0$. Recall that, by assumption,  $\Delta_0$ and $- \sigma \Delta_0$ are disconnected in the Dynkin diagram of $\Delta$. Since $\supp(\alpha)$ is connected, either $\supp(\alpha) \subseteq \Delta_0$ or $\supp(\alpha) \subseteq - \sigma \Delta_0$. Therefore, $\alpha \in \RR_0 \sqcup \sigma \RR_0$, which contradicts $\alpha \in \RR_1$.
    Conversely, suppose that $\supp(\alpha) \cap \Delta_1 \neq \emptyset$. Then $\alpha$ does not belong to $\RR_0 \sqcup \sigma \RR_0$, so $\alpha \in \RR_1$.

    Now, to show that $\RR$ is closed, fix $\alpha, \beta \in \RR$, and suppose that $\alpha + \beta$ is a root. If $\alpha, \beta \in \RR_0$, then clearly $\alpha + \beta \in \RR_0$.
    Suppose  $\alpha \in \RR_1$. By (\ref{RR1 condition}), there exists $\varphi \in \Delta_1$ such that $c_\varphi(\alpha) \geq 1$. If $\beta$ is positive, then $c_\varphi(\alpha + \beta) \geq c_\varphi(\alpha) \geq 1$. If $\beta \in \RR_0$, then $c_\varphi(\beta) = 0$, so $c_\varphi(\alpha + \beta) = c_\varphi(\alpha) \geq 1$. In either case, it follows that $\alpha + \beta \in\RR_1$, thanks to (\ref{RR1 condition}).

    Next, let us show that $\RR \sqcup \sigma \RR = \Phi$. Observe that 
    $$\RR \cap \sigma \RR = (\RR_1 \cap \sigma \RR_1) \sqcup (\RR_1 \cap \sigma \RR_0) \sqcup (\RR_0\cap \sigma \RR_1)\sqcup (\RR_0\cap \sigma \RR_0).$$
    Since $\Delta_0 \sqcup (-\sigma \Delta_0)$ is a linearly independent subset of $(\h^\C)^*$, it follows that $\C \Delta_0 \cap \C (- \sigma \Delta_0) = 0$, so $\RR_0\cap \sigma \RR_0 = \emptyset$.
    Since $\RR_1 = \Phi^+ \backslash (\RR_0 \sqcup  \sigma \RR_0)$, we have $\sigma \RR_1 = - \RR_1$. By definition, we have $\emptyset = \RR_1 \cap \sigma \RR_1 = \RR_1 \cap \sigma \RR_0 = \RR_0\cap \sigma \RR_1$. Therefore, $\RR \cap \sigma \RR  = \emptyset$.  The union $\RR \sqcup \sigma \RR = \Phi$ follows because  $|\RR| = \frac12 |\Phi|.$ 

    Finally, let us show that $H_\alpha \in \l$ for all $\alpha  \in \RR \cap (- \RR)$; it suffices to show that $\RR_0 = \RR \cap (-\RR)$. It is obvious that $\RR_0 \subseteq \RR \cap (-\RR)$. For the opposite inclusion, suppose $\alpha \in \RR \cap (-\RR)$. Since $\sigma \RR_1 = - \RR_1$,  (A2) implies that $(- \RR_1) \cap \RR = \emptyset$. Thus,  $\alpha \in \RR_0$. This completes the proof of (A).

    Let us show (B). Using the notation  of Proposition \ref{regular structure prop} Part (B), set
    $\RR_0 := \RR \cap (-\RR)$ and $ \RR_1 := \RR \backslash \RR_0.$ It suffices to show the following three facts:
    \begin{enumerate}
        \item[(B1)] Let $\RR^+_0$ be a set of positive roots for $\RR_0$. Then 
        $$\Phi^+ := \RR_0^+ \sqcup - \sigma \RR^+_0 \sqcup \RR_1$$
        is a set of positive roots for $\Phi$ such that $\sigma \Phi^+ = - \Phi^+$.
        \item[(B2)] If $\Delta_0$ denotes the simple roots of $\RR^+_0$ and $\Delta$ denotes the simple roots of $\Phi^+$, then $\Delta_0 \subseteq \Delta$.
        \item[(B3)] The subdiagrams corresponding to  $\Delta_0$ and $- \sigma \Delta_0$  are disconnected in the Dynkin diagram of $\Delta$.
    \end{enumerate}
    Indeed, suppose that (B1), (B2) and (B3) hold. (B1) implies that $\sigma:\Phi \rightarrow \Phi$ has no fixed points. Thus, $\h$ is maximally compact, since there are no real roots. Parts (B1), (B2) and (B3) imply that $(\Phi^+, \Delta_0, \l)$ satisfy the conditions described in Algorithm \ref{regular complex structures construction algorithm}. Finally, observe that $\RR_0 = \C \Delta_0 \cap \Phi$, $\RR_0^+ = \RR_0 \cap \Phi^+$, and $\RR = (-\RR_0^+) \sqcup \Phi^+ \backslash( - \sigma \RR_0^+),$ so $\q$ is one of the subalgebras constructed in Part (A). 

    Let us show (B1). We begin by showing the following two facts:
    \begin{enumerate}[(i)]
        \item $\sigma \RR_1 = - \RR_1$.
        \item If $\alpha \in \RR_0$ and $\beta \in \sigma \RR_0$, then $\alpha + \beta$ is not a root.
    \end{enumerate}
    To show (i), observe that 
    $$\RR_0 \sqcup \RR_1 \sqcup \sigma \RR_0 \sqcup  \sigma \RR_1 = \Phi = - \Phi = \RR_0 \sqcup (-\RR_1) \sqcup \sigma \RR_0 \sqcup  (-\sigma \RR_1). $$
    It is easy to see that  $\emptyset = \RR_1 \cap \RR_0 = \RR_1 \cap (-\RR_1) = \RR_1 \cap \sigma \RR_0$. Therefore, it follows that $\RR_1 \subseteq - \sigma \RR_1$, so $\RR_1 = - \sigma \RR_1$, since these two sets have the same cardinality. To show (ii), fix $\alpha \in \RR_0$, $\beta \in \sigma \RR_0$, and for the sake of contradiction, suppose $\alpha + \beta$ is a root. Without loss of generality, suppose that $\alpha + \beta \in \RR$. Since $\RR$ is closed,  $\beta = (\alpha + \beta) - \alpha \in \RR$, a contradiction. Now, to show (B1), fix a set of positive roots $\RR^+_0$ of $\RR_0$, and define $\Phi^+$ as in (B1). First, let us show that $\Phi^+$ is a closed subset. Fix $\alpha, \beta \in \Phi^+$, and suppose $\alpha + \beta$ is a root.
    \begin{enumerate}
        \item If $\alpha,\beta \in \RR^+_0$, then $\alpha + \beta \in \RR^+_0$, since $\RR_0$ is a closed subset and $\RR^+_0$ is a set of positive roots for $\RR_0$. Thus, if $\alpha, \beta \in - \sigma \RR^+_0$, then $\alpha + \beta \in - \sigma \RR^+_0$.
        \item If $\alpha \in \RR_1$  and $\beta \in \RR $, then  Remark \ref{closed subsets prop} implies that $\alpha + \beta \in \RR_1$. 
        \item If $\alpha \in \RR_1$ and $\beta \in - \sigma \RR_0^+$, then (i) implies that $-\sigma \alpha\in \RR_1$ and $- \sigma \beta \in \RR_0$. Thus, Remark \ref{closed subsets prop} implies that $-\sigma (\alpha + \beta) \in \RR_1$, so $\alpha + \beta \in \RR_1$ by (i).
        \item Fact (ii) implies that the case $\alpha \in \RR_0^+$ and $\beta \in -\sigma \RR_0^+$ cannot happen.
    \end{enumerate}
    In all cases, we find that $\alpha + \beta \in \Phi^+$.  It is easy to see that  $\Phi^+ \sqcup (- \Phi^+) = \Phi$.  The equality $\sigma \Phi^+ = - \Phi^+$ follows from (i).

    To show (B2), fix $\gamma \in \Delta_0$. For the sake of contradiction, suppose $\gamma$ is not simple in $\Phi^+$, i.e. $\gamma = \alpha + \beta$ for some $\alpha, \beta \in \Phi^+$. From the previous paragraph, the only possibility is that $\alpha, \beta \in \RR_0^+$, but this contradicts  the simplicity of $\gamma$  in $\RR^+_0$.

    Finally, let us show (B3). Since $\RR \cap \sigma \RR = \emptyset$, it follows that $\Delta_0 \cap (- \sigma \Delta_0) = \emptyset$. For the sake of contradiction, suppose that $\alpha \in \Delta_0$ and $\beta \in -\sigma \Delta_0$ share an edge. By definition, this means that $\langle \alpha, \beta \rangle < 0$. Thus, $\alpha + \beta$ is a root by \cite[Proposition 2.48]{knapp}. However, (ii) tells us that $\alpha + \beta$ is not a root.

\section{The balanced condition}
\label{balanced section}
In this section, let $G$ be a real semisimple Lie group with Lie algebra $\g$. Let $\h$ be a maximally compact Cartan subalgebra of $\g$, and let $\Phi$ denote the roots of $(\g^\C, \h^\C)$. Let $\t$ be as in (\ref{t and a definitions}), and let $T$ be the connected subgroup of $G$ with Lie algebra $\t$. Let $J = J (\Phi^+, \Delta_0, \l)$ be an $\h$-regular complex structure (see Proposition \ref{regular structure prop}). 

The goal of this section is to prove Theorem \ref{balanced technical theorem}. Observe that $J$ is right $T$-invariant (see \S \ref{symmeterisation 2 subsection}). Thus, by symmetrisation (Proposition \ref{symmetrizing to right T invariant} and Lemma \ref{AdGT is compact}), it suffices to consider the balanced equation for left-invariant right $T$-invariant Hermitian metrics on $(G,J)$.

\subsection{Left-invariant right $T$-invariant balanced metrics}
By Lemma \ref{sigma adapted root vectors lemma}, there exist non-zero root vectors $E_\alpha \in \g^\C_\alpha$ such that 
\begin{equation}
    \label{sigma adapted root vectors equations}
    [E_\alpha, E_{-\alpha}] = H_\alpha, \qquad \sigma E_\alpha = \begin{cases}
    E_{\sigma \alpha} & \text{if $\alpha$ is complex, } \\ -s_\alpha E_{-\alpha} & \text{if $\alpha$ is imaginary},
\end{cases}
\end{equation}
where $H_\alpha$ is the dual of $\alpha$ with respect to the Killing form $B$  of $\g^\C$ restricted to $\h^\C$, and
$$s_\alpha := \begin{cases}
    1 & \text{if $\alpha$ is compact,} \\ -1 & \text{if $\alpha$ is non-compact.}
\end{cases}$$
Let $\varepsilon^\alpha: \g^\C \rightarrow \C$ denote the dual of $E_\alpha$ with respect to $\g^\C = \h^\C \oplus \bigoplus_{\alpha \in \Phi} \g^\C_\alpha$. 

By definition of a regular complex structure, the $\i$-eigenspace of $J = J(\Phi^+, \Delta_0, \l)$ is 
$$\q = \l \oplus \bigoplus_{\alpha \in \RR} \g^\C_\alpha, \qquad \RR = \RR_0 \sqcup \Phi^+_{\Im} \sqcup \RR_2,$$
where $\RR_0 = \R \Delta_0 \cap \Phi$, $\Phi^+_{\Im}$ are the imaginary positive roots, and $\RR_2$ is the set of complex roots in $\Phi^+ \backslash (\RR_0^+ \sqcup  (- \sigma \RR_0^+))$.  Using Proposition \ref{right T invariant metrics form}, it is easy to see the following:

\begin{lemma}
    \label{right T invariant Hermitian metrics form}
    Let $g$ be a left-invariant right $T$-invariant Hermitian metric on $(G,J)$, and let $\omega := g(J \cdot,\cdot)$ be the corresponding positive $(1,1)$-form. Then 
    $$\omega = \omega|_{\h \times \h} + \i \sum_{\alpha \in \RR} \lambda_\alpha \varepsilon^\alpha \wedge \sigma \varepsilon^\alpha +\i \sum_{\alpha \in \RR_2} \mu_\alpha \varepsilon^\alpha \wedge \varepsilon^{-\alpha}, $$
    where 
    \begin{enumerate}[\normalfont (i)]
        \item  $\omega|_{h \times \h}$ is a positive $(1,1)$-form on $(\h, J|_{\h})$,
        \item $\lambda_\alpha = g(E_\alpha, \sigma E_\alpha)$ is a positive real number for any $\alpha \in \RR$, and 
        \item $\mu_\alpha = g(E_\alpha, E_{- \alpha})$ for $\alpha \in \RR_2$ are complex numbers satisfying $\mu_{- \sigma \alpha} = \overline{\mu_\alpha}$ and $\lambda_\alpha \lambda_{-\sigma \alpha} - |\mu_\alpha|^2 > 0.$
    \end{enumerate}
    Conversely, any such $\omega$ defines a left-invariant right $T$-invariant Hermitian metric on $(G,J)$.
\end{lemma}

Decompose $\Phi^+_{\Im} = \Phi^+_{\c} \sqcup \Phi^+_{\nc}$, where $\Phi^+_{\c}$ and $\Phi^+_{\nc}$ are the compact and non-compact positive roots, respectively. Observe that $- \sigma \RR_2 = \RR_2$, so $-\sigma$ partitions $\RR_2$ into two-element orbits, and we can choose a subset $\RR_3 \subseteq \RR_2$ such that $\RR_3 \sqcup (-\sigma \RR_3) = \RR_2$.

\begin{proposition}
    \label{balanced equation prop}
    Using the notation of Lemma \ref{right T invariant Hermitian metrics form}, a left-invariant right $T$-invariant Hermitian metric $g$ on $(G,J)$ is balanced if and only if 
    \begin{equation}
        \label{balanced equation}
        \sum_{\alpha \in \Phi^+_{\c}} \frac{\alpha}{\lambda_\alpha} - \sum_{ \alpha \in \Phi^+_{\nc}} \frac{\alpha}{\lambda_\alpha} + \sum_{\alpha \in \RR_2} \frac{\mu_{- \sigma \alpha } \alpha}{D_\alpha} = 0,
    \end{equation}
    where $D_\alpha := \lambda_\alpha \lambda_{- \sigma \alpha} - |\mu_\alpha|^2 > 0$.
\end{proposition}
\begin{proof}
    Let $\{V_1,\ldots, V_n\}$ be an orthonormal basis of $\q$ with respect to the Hermitian inner product $g(\cdot, \sigma \cdot): \q \times \q \rightarrow \C$. By \cite[Lemma 2.1]{Andrada}, $g$ is balanced if and only if 
    $$\sum [V_i, \sigma V_i] = 0.$$ 
    By the Gram-Schmidt process, an orthonormal basis $\{V_i\}$ for $\q$ is given by
    $$\{H_1,\ldots, H_r\} \sqcup \left\{\frac{1}{\sqrt{\lambda_\alpha}}E_\alpha : \alpha \in \RR_0 \sqcup \Phi^+_{\Im}\right\} \sqcup \left\{Z_\alpha^{(1)}, Z_\alpha^{(2)} : \alpha \in \RR_3\right\}, $$
    where $H_1,\ldots, H_r$ is an orthonormal basis for $\l$, and 
    $$Z_\alpha^{(1)} = \frac{1}{\sqrt{\lambda_\alpha}}E_\alpha, \qquad Z_\alpha^{(2)} = \frac{\sqrt{\lambda_\alpha}}{\sqrt{D_\alpha}} \left( E_{- \sigma \alpha} - \frac{\overline{ \mu_\alpha}}{ \lambda_\alpha} E_\alpha\right).$$
    We find 
    \begin{align*}
        \sum[V_i, \sigma V_i] &= \sum \underbrace{[H_i, \sigma H_i]}_{ = 0} + \sum_{\alpha \in \RR_0} \frac{1}{\lambda_\alpha}  \underbrace{[E_\alpha, \sigma E_\alpha]}_{=0}  + \sum_{\alpha \in \Phi^+_{\Im}} \frac{1}{\lambda_\alpha}  [E_\alpha, \sigma E_\alpha] + \\
        & \qquad + \sum_{\alpha \in \RR_3} \frac{1}{\lambda_\alpha}  \underbrace{[E_\alpha, \sigma E_\alpha]}_{=0} + \sum_{\alpha \in \RR_3} [Z^{(2)}_\alpha, \sigma Z^{(2)}_\alpha] \\
        &= -  \sum_{\alpha \in \Phi^+_{\Im}} \frac{s_\alpha}{\lambda_\alpha} H_\alpha - \sum_{\alpha \in \RR_3} \left( \mu_{-\sigma \alpha} H_\alpha + \mu_\alpha H_{- \sigma \alpha}\right).
    \end{align*}
    The first equality holds because $\alpha + \sigma \alpha$ is real for any root $\alpha$, so $\alpha + \sigma \alpha$ cannot be a root since $\h$ is maximally compact. In particular, $[\g^\C_\alpha, \g^\C_{\sigma \alpha}] = 0$ for complex $\alpha$. The second equality follows by the properties of the root vectors $E_\alpha$  (\ref{sigma adapted root vectors equations}). 
    Transferring  the above to $(\h^\C)^*$ via the Killing form gives Equation (\ref{balanced equation}).
\end{proof}

Observe that Equation (\ref{balanced equation}) does not depend on  the choice of $\l$  nor $g|_{\h\times \h}$. Therefore, the left-invariant balanced condition on $(G,J)$ is determined by $(\V, \Delta_0)$, where $\V$ is the Vogan diagram determined by $\Phi^+$.

\subsection{Balanced pairs $(\V, \Delta_0)$} 
\label{balanced pairs subsection}Let $\V$ be an abstract Vogan diagram.   By Proposition \ref{Vogan diagrams correspondence}, $\V$ determines a root system $\Phi$, a complex conjugation map $\sigma: \Phi \rightarrow \Phi$, a decomposition of the roots into complex, compact and non-compact, and a set of positive roots $\Phi^+$ such that $\sigma \Phi^+ = - \Phi^+$.

Let $\Delta_0$ be a subset of the vertices of $\V$ such that $\Delta_0$ and $- \sigma \Delta_0$ are disconnected in the underlying Dynkin diagram of $\V$.
Let $\RR_0 := \R \Delta_0 \cap \Phi$, and let $\RR_2$ be the complex roots in $\Phi^+ \backslash (\RR_0^+ \sqcup (-\sigma \RR_0^+))$.

\begin{definition}
    We say that the pair $(\V, \Delta_0)$ is \textit{balanced} if there exists a solution to Equation (\ref{balanced equation}) for some $\lambda_\alpha > 0$, $\alpha \in \RR$ and $\mu_\alpha \in \C$, $\alpha \in \RR_2$ satisfying $\mu_{-\sigma \alpha} = \overline{\mu_\alpha}$ and $D_\alpha := \lambda_\alpha \lambda_{-\sigma \alpha} - |\mu_\alpha|^2 > 0$.
\end{definition}

Let $G$ be a real semisimple Lie group, let $J = J(\Phi^+, \Delta_0, \l)$ be a regular complex structure, and let $\V$ be the Vogan diagram determined by $\Phi^+$. Decompose $\V = \V^{(1)}\sqcup \cdots \sqcup \V^{(s)}$, where $\V^{(i)}$ are the connected components of $\V$, and let $\Delta_0 = \Delta_0^{(1)} \sqcup \cdots \sqcup \Delta_0^{(s)}$ be the corresponding decomposition of $\Delta_0$.
By Proposition \ref{balanced equation prop}, the following are equivalent:
\begin{enumerate}[(i)]
    \item $(G,J)$ admits a left-invariant balanced metric.
    \item The pair $(\V, \Delta_0)$ is balanced.
    \item Each $(\V^{(i)}, \Delta_0^{(i)})$ is balanced for $i = 1,\ldots, s$.
\end{enumerate}
Thus, Theorem \ref{balanced technical theorem} is proved if we show the following:
\begin{proposition}
    \label{balanced pairs prop}
    Let $\V$ be connected, and let $\Delta_0$ be as above. Then $(\V, \Delta_0)$ is balanced if and only if $\V$ is not listed in Table \ref{vogan table}.
\end{proposition}

We dedicate the rest of this section to proving Proposition \ref{balanced pairs prop}.  Henceforth, assume that $\V$ is a connected Vogan diagram. Recall that $\V$ is either inner, complex or is listed in Table \ref{mixed diagrams}. 

\begin{lemma}
    If $\V$ is connected but not inner, then $(\V, \Delta_0)$ is balanced for any choice of $\Delta_0$.
\end{lemma}
\begin{proof}
    If $\V$ is complex, then all roots are complex, so $\Phi^+_{\c} = \Phi^+_{\nc} = \emptyset$. Setting $\lambda_\alpha > 0$ to be arbitrary and $\mu_\alpha = 0$ for $\alpha \in \RR_2$ gives a solution to Equation (\ref{balanced equation}).

    Next, suppose $\V$ is listed in  Table \ref{mixed diagrams}. If we choose $\mu_\alpha$ to be real, we can rewrite Equation (\ref{balanced equation}) as
    $$\sum_{\alpha \in \Phi^+_{\c}} \frac{\alpha}{\lambda_\alpha} - \sum_{ \alpha \in \Phi^+_{\nc}} \frac{\alpha}{\lambda_\alpha} + \sum_{\alpha \in \RR_2} \frac{\mu_{ \alpha } }{ 2D_\alpha} (\alpha - \sigma \alpha) = 0.$$
    Observe that any real number can be written as $\mu_{ \alpha } /  2D_\alpha$  for some choice of $\mu_\alpha, \lambda_\alpha, \lambda_{-\sigma \alpha}$. Thus, it suffices to show that 
    \begin{equation}
        \label{imaginary vector contained in V}
        \sum_{\alpha \in \Phi^+_{\Im}} \frac{s_\alpha \alpha}{ \lambda_\alpha}  \in V, \qquad V = \R\left\{ \alpha - \sigma \alpha : \alpha \in \RR_2\right\}.
    \end{equation}
    for some choice of $\lambda_\alpha$, $\alpha \in \Phi^+_{\Im}$.
    For each Vogan diagram in Table \ref{mixed diagrams}, we give a solution to (\ref{imaginary vector contained in V}).

    First, suppose $\V$ is the following Vogan diagram with $k \geq 1$:
        $$\dynkin[edge length=0.8cm,labels={\alpha_1,\alpha_2,\alpha_k,\alpha_{k+1}, \alpha_{2k - 1}, \alpha_{2k}},
          label directions={above,above,above,above, above, above}, root radius=.1cm, involutions={16;25;34} ]{A}{oo.oo.oo}.$$
        Observe that $$\RR_0^+ \sqcup (- \sigma \RR_0^+) \subseteq  \left\{ \alpha_i + \cdots + \alpha_j  : 1 \leq i \leq j \leq k - 1 \text{ or } k+2 \leq i \leq j \leq 2k\right\}.$$
        Thus, $\alpha_j + \cdots +\alpha_k \in \RR_2$ for $1 \leq j\leq k$, so $\alpha_{k+ 1 - i} + \cdots + \alpha_{k+ i} \in V$ for any $i = 1,\ldots,k$. It follows that $\alpha_{k+ 1 - i} + \alpha_{k+ i} \in V$ for any $i = 1,\ldots, k$. Since the span of $\alpha_{k+ 1 - i} + \alpha_{k+ i}$ contains all imaginary roots, any choice of $\lambda_\alpha > 0$, $\alpha \in \Phi^+_{\Im}$ satisfies (\ref{imaginary vector contained in V}).
        
    Suppose $\V$ is the following Vogan diagram with $k \geq 2$: 
        $$\dynkin[edge length=0.8cm,labels={\alpha_1,\alpha_2,\alpha_{k-1},\alpha_{k}, \alpha_{k+1}, \alpha_{2k-2}, \alpha_{2k-1}},
          label directions={above,above,above,above, above, above, above}, root radius=.1cm, involutions={17;26;35} ]{A}{oo.oto.oo},$$
          where
          \dynkin[edge length=1cm, root radius=.1cm]{A}{t} denotes possibly either \dynkin[edge length=1cm, root radius=.1cm]{A}{*} or \dynkin[edge length=1cm, root radius=.1cm]{A}{o}. Observe that 
          $$\RR_0^+ \sqcup (- \sigma \RR_0^+) \subseteq  \left\{ \alpha_i + \cdots + \alpha_j  : 1 \leq i \leq j \leq k - 1 \text{ or } k+1 \leq i \leq j \leq 2k -1\right\}.$$
          Thus, $\alpha_{k-i} + \cdots + \alpha_k \in \RR_2$ for $0 \leq i \leq k-1$, so  $\alpha_{k-i} + \cdots + \alpha_{k-1} + 2 \alpha_k + \alpha_{k+1} + \cdots + \alpha_{k+i} \in V$. Thus, 
          $$\alpha_{k-1} + 2 \alpha_k + \alpha_{k+1} \in V, \qquad \alpha_{k-i} + \alpha_{k+i} \in V \text{ for $i \geq 2$}.$$
          The positive imaginary roots are precisely $\alpha_{k-i} + \cdots + \alpha_{k-1} + 2 \alpha_k + \alpha_{k+1} + \cdots + \alpha_{k+i}$ for $0 \leq i \leq k-1$. By \cite[Proposition 6.104]{knapp}, $s_\alpha = s_{\alpha_k}$ for all imaginary roots $\alpha$. Set $\lambda_{\alpha_k} = 1/ (k-1)$, and set $\lambda_\alpha = 1$ for $\alpha \neq \alpha_k$. Then 
          $$\sum_{\alpha \in \Phi^+_{\Im}} \frac{ \alpha}{ \lambda_\alpha} =   \alpha_1 + 2  \alpha_2 +  \cdots + (k-1) \alpha_{k-1} + 2(k-1) \alpha_k + (k-1) \alpha_{k+1} + \cdots  +  2 \alpha_{2k- 2}+\alpha_{2k-1} \in V. $$

    Suppose $\V$ is the following Vogan diagram with $r \geq 4$: 
    $$\dynkin[edge length=1cm, labels={\alpha_1,\alpha_2,\alpha_{r-2},\alpha_{r-1}, \alpha_{r}}, root radius=.1cm, involutions={*[relative, out=0, in=180]45}]{D}{tt.too}$$
    where
    \dynkin[edge length=1cm, root radius=.1cm]{A}{t} denotes possibly either \dynkin[edge length=1cm, root radius=.1cm]{A}{*} or \dynkin[edge length=1cm, root radius=.1cm]{A}{o}. Observe that $\RR_0^+ \sqcup (- \sigma \RR_0^+) \subseteq \{\alpha_r, \alpha_{r-1}\}$.  Thus, $\alpha_i + \cdots + \alpha_{r-1} \in \RR_2$ for all $1 \leq i \leq r-2$, so $2(\alpha_i + \cdots + \alpha_{r-2}) + \alpha_{r-1} + \alpha_r \in V$. It follows that 
    $$2 \alpha_{r-2} + \alpha_{r-1} + \alpha_r \in V, \qquad \alpha_i \in V \text{ for } 1 \leq i \leq r-3.$$
    Let $s_i := s_{\alpha_i}$. By \cite[Proposition 6.104]{knapp} and Corollary \ref{compactness addition rules}, the positive imaginary roots and their compactness are given by the following table:
    $$\begin{tabular}{c|c|c}
             $\alpha$ &   $s_\alpha $ & Conditions\\\hline
             $\alpha_i + \cdots + \alpha_j$ & $s_i s_{i+1} \cdots s_j$ & $1 \leq i \leq j \leq r-2$ \\
             $\alpha_i + \cdots + \alpha_{r-2} + \alpha_{r-1} + \alpha_r$ & $s_i s_{i+1} \cdots s_{r-2}$ & $1 \leq i \leq r-2$ \\
             $\alpha_i + \cdots + \alpha_{j-1} + 2(\alpha_j + \cdots + \alpha_{r-2}) + \alpha_{r-1} + \alpha_r$ & $s_i s_{i+1} \cdots s_{j-1}$ & $1 \leq i < j\leq r-2$ \\
    \end{tabular}$$
    Set $\lambda_\alpha = 1$ for all imaginary positive $\alpha$. If $1 \leq i \leq j \leq r-3$, then  $\alpha_i + \cdots + \alpha_j \in V$. Moreover, $\alpha_i + \cdots + \alpha_{j-1} + 2(\alpha_j + \cdots + \alpha_{r-2}) + \alpha_{r-1} + \alpha_r \in V$ for $1 \leq i < j\leq r-2$. Finally, $\alpha_i + \cdots + \alpha_{r-2}$ and $\alpha_i + \cdots + \alpha_{r-2} + \alpha_{r-1} + \alpha_r$ have the same compactness, and their sum belongs to $V$. Thus, $\sum_{\alpha \in \Phi^+_{\Im}} s_\alpha \alpha \in V$.

    Finally, suppose $\V$ is the following Vogan diagram:
    $$\dynkin[edge length=1cm, labels={\alpha_1,\alpha_2,\alpha_3,\alpha_4, \alpha_5, \alpha_6}, label directions={above,right,above,above right, above, above}, root radius=.1cm, involutions={16;35}]{E}{ototoo}$$
    where
    \dynkin[edge length=1cm, root radius=.1cm]{A}{t} denotes possibly either \dynkin[edge length=1cm, root radius=.1cm]{A}{*} or \dynkin[edge length=1cm, root radius=.1cm]{A}{o}. Observe that 
    $$\RR_0^+ \sqcup (- \sigma \RR_0^+) \subseteq \{\alpha_1, \alpha_3, \alpha_1 +\alpha_3, \alpha_5, \alpha_6, \alpha_5 + \alpha_6\}.$$
    Thus, the following roots belong to $\RR_2$:
    $$\alpha_3 + \alpha_4, \quad \alpha_1 + \alpha_3 + \alpha_4, \quad \alpha_1 + \alpha_3 + \alpha_4 + \alpha_2, \quad \alpha_1 + 2 \alpha_3 + 2 \alpha_4 + \alpha_5 + \alpha_2.$$
    It follows that the following vectors belong to $V$:
    $$\alpha_3 + 2 \alpha_4 + \alpha_5, \quad \alpha_1 + \alpha_3 + 2 \alpha_4 + \alpha_5 + \alpha_6, \quad \alpha_1 + \alpha_3 + 2 \alpha_4 + \alpha_5 + \alpha_6 + 2 \alpha_2,$$
    $$\alpha_1 + 3 \alpha_3 + 4 \alpha_4 + 3 \alpha_5 + \alpha_6 + 2 \alpha_2.$$
    It follows that $\alpha_1 + \alpha_6, \alpha_3 + \alpha_5, \alpha_4, \alpha_2 \in V$. Since the span of these vectors contains all imaginary roots, any choice of $\lambda_\alpha > 0$, $\alpha \in \Phi^+_{\Im}$ satisfies (\ref{imaginary vector contained in V}). 
\end{proof}

It remains to consider when $\V$ is inner, i.e. the involution of $\V$ is trivial. Henceforth, let us assume that this is the case. Note that the only possible choice of $\Delta_0$ is $\Delta_0 = \emptyset$. Let $\alpha_1,\ldots, \alpha_r$ denote simple roots of $\Phi^+$, which we identify with the vertices of $\V$. For each $\alpha \in \Phi^+$, let us write 
$$\alpha = \sum_{i=1}^r c_i(\alpha) \alpha_i.$$
Let $A$ be the matrix whose rows are indexed by $\alpha \in \Phi^+$, and whose $i$th row is $$s_\alpha(c_1(\alpha), \ldots, c_r(\alpha)),$$
where $s_\alpha = 1$ if $\alpha$ is compact, and $s_\alpha = -1$ if $\alpha$ is non-compact.

\begin{lemma}
    \label{stiemke lemma}
    Suppose $\V$ is connected and inner. Then $(\V, \emptyset)$ is balanced if and only if there does not exist $y \in \R^r$ with $Ay \geq 0$ and $Ay \neq 0$.
\end{lemma}
Here, $Ay \geq 0$ means that each entry of $Ay$ is non-negative.
\begin{proof}
    First, from Equation (\ref{balanced equation}), observe that $(\V, \emptyset)$ is balanced if and only if there exist $x_\alpha > 0$ for $\alpha \in \Phi^+$ such that 
    $$\sum_{\alpha \in \Phi^+} x_\alpha s_\alpha \alpha = 0.$$
    Thus, $(\V, \emptyset)$ is balanced if and only if there exists $x \in \R^{|\Phi^+|}$ with $x > 0$ such that $A^\top x = 0$. The result follows from Stiemke's theorem \cite[Chapter 2.4, Theorem 7]{mangasarian}.
\end{proof}

\begin{remark}
    \label{Ay expansion}
    Suppose $y \in \R^r$ and $Ay \geq 0$. The $\alpha$th entry of $Ay$ is $s_\alpha \sum c_i(\alpha) y_i$. Thus, if $\alpha$ is compact, then $\sum c_i(\alpha) y_i \geq 0$, and if $\alpha$ is non-compact, then $\sum c_i(\alpha) y_i \leq 0$. In particular, if $\alpha_k$ is compact then $y_k \geq 0$, and  if $\alpha_k$ is non-compact then  $y_k \leq 0$.
\end{remark}

\begin{lemma}
    \label{yk lemma}
    Suppose $\V$ has at least one painted vertex. Suppose $y \in \R^r$ satisfies $Ay \geq 0$ and $Ay \neq 0$. Then there exists $k = 1,\ldots, r$ such that $y_k < 0$.
\end{lemma}
\begin{proof}
        Since $Ay \neq 0$, we know that $y \neq 0$. Therefore, there exists some $i_1 = 1,\ldots, r$ such that $y_{i_1} \neq 0$. If $y_{i_1} < 0$, we are done. Suppose $y_{i_1} > 0$. In particular, $\alpha_{i_1}$ is compact by Remark \ref{Ay expansion}. Choose a path in the Dynkin diagram $\alpha_{i_1}, \alpha_{i_2},\ldots, \alpha_{i_\ell}$ such that $\alpha_{i_\ell}$ is non-compact, but $\alpha_{i_1},\ldots, \alpha_{i_{\ell - 1}}$ are compact. By Corollary \ref{compactness addition rules} and Lemma \ref{root chain lemma}, $\alpha_{i_1} + \cdots + \alpha_{i_{\ell - 1}}$ is a compact positive root and $\alpha_{i_1} + \cdots + \alpha_{i_{\ell }}$ is a non-compact positive root. Thus, Remark \ref{Ay expansion} implies that 
        $$\underbrace{y_{i_1} + \cdots + y_{i_{\ell - 1}}}_{ > 0} + y_{i_\ell} \leq 0,$$
        so $y_{i_\ell} < 0$, as desired.
\end{proof}

We are now ready to finish the proof of Proposition \ref{balanced pairs prop}. Again, assume that $\V$ is connected and inner. 

First, suppose $\V$ has no painted vertices. Then all roots are compact, so $s_\alpha = 1$ for all $\alpha \in \Phi^+$. Thus, $y = (1,\ldots,1)^\top$ satisfies $Ay \geq 0$ and $Ay \neq 0$, so $(\V, \emptyset)$ is not balanced by Lemma \ref{stiemke lemma}.

Next, suppose $\V$ has exactly one painted vertex $\alpha_k$, and $\V$ is listed in Table \ref{vogan table}. By Lemma \ref{single painted vertex lemma}, it follows that $c_k(\alpha) \in \{0,1\}$ for all $\alpha \in \Phi^+$. Since $\alpha_k$ is the only non-compact simple root, $\alpha \in \Phi^+$ is compact if and only if $c_k(\alpha) = 0$. Let $y = (0,\ldots, 0 ,-1,0,\ldots,0)^\top$, where the $-1$ is in the $k$th spot. For each $\alpha \in \Phi^+$, the $\alpha$th entry of $Ay$ is 
    $$s_\alpha \sum_{i= 1}^r c_i(\alpha) y_i = -s_\alpha c_k(\alpha) = \begin{cases}
        0 & \text{if $c_k(\alpha) = 0$,} \\
        1 & \text{if $c_k(\alpha) = 1$.}
    \end{cases}$$
    Thus, $Ay \geq 0$ and $Ay \neq 0$, so $(\V, \emptyset)$ is not balanced.

Now, suppose that $\V$ has exactly one painted vertex $\alpha_k$, but $\V$ is not listed in  Table \ref{vogan table}. For the sake of contradiction, suppose that $(\V, \emptyset)$ is not balanced. By Lemma \ref{stiemke lemma}, there exists $y \in \R^r$ such that $Ay \geq 0$ and $Ay \neq 0$. By Lemma \ref{single painted vertex lemma}, there exists a non-compact positive root $\alpha$ such that $\alpha_k + \alpha$ is a root. Since $\alpha_k$ is the only painted vertex, Lemma \ref{yk lemma} and Remark \ref{Ay expansion} imply that $y_k < 0$. Since $\alpha_k + \alpha$ is compact, Remark \ref{Ay expansion} implies that $$0 \leq \sum_{i=1}^r c_i(\alpha_k + \alpha) y_i =  \underbrace{y_k}_{ < 0} + \underbrace{\sum_{i=1}^r c_i(  \alpha) y_i}_{\leq 0} < 0,  $$
    a contradiction.

Finally, suppose $\V$ has at least $2$ painted vertices. For the sake of contradiction, suppose that $(\V, \emptyset)$ is not balanced. Then there exists $y \in \R^r$ such that $Ay \geq 0$ and $Ay \neq 0$. By Lemma \ref{yk lemma}, there exists $i_1 = 1,\ldots, r$ such that $y_{i_1} < 0$; in particular, $\alpha_{i_1}$ is non-compact. Choose a path in the Dynkin diagram $\alpha_{i_1}, \alpha_{i_2},\ldots, \alpha_{i_\ell}$ such that $\alpha_{i_\ell}$ is also non-compact, but $\alpha_{i_2},\ldots, \alpha_{i_{\ell -1}}$ are compact.
    By Corollary \ref{compactness addition rules} and Lemma \ref{root chain lemma},  $\alpha_{i_2} + \cdots + \alpha_{i_{\ell}}$ is a non-compact positive root, and $\alpha_{i_1} + \cdots + \alpha_{i_{\ell}}$ is a compact positive root. 
Thus, Remark \ref{Ay expansion} implies that 
$$0 \leq \underbrace{y_{i_1} }_{ <  0} + \underbrace{y_{i_2} + \cdots + y_{i_\ell}}_{ \leq 0} < 0,$$
a contradiction.

\section{The pluriclosed condition}
\label{pluriclosed section}
In this section, let $G$ be a real semisimple Lie group with Lie algebra $\g$. Let $\h$ be a maximally compact Cartan subalgebra of $\g$, and let $\Phi$ be the roots of $(\g^\C, \h^\C)$. Let $J = J(\Phi^+, \Delta_0, \l)$ be a regular complex structure on $G$ (see Proposition \ref{regular structure prop}). Let $\t \leq \h$ be the subspace defined in (\ref{t and a definitions}), and let $T$ denote the connected subgroup of $G$ with Lie algebra $\t$.

Let $g$ denote a left-invariant right $T$-invariant Hermitian metric on $(G,J)$, and let $\omega := g(J \cdot,\cdot)$ be the corresponding positive $(1,1)$-form. Let $d^c$ be the operator on forms given by $d^c := J \circ d \circ J$, where $J$ acts on forms via pullback. In particular, $d^c \omega = d \omega(J \cdot, J \cdot, J \cdot)$. We have $dd^c \omega = -2 \i \partial \overline \partial \omega$, so $\omega$ is pluriclosed if and only if $dd^c \omega = 0$.
The following proves Part (a) of Theorem \ref{pluriclosed technical theorem}:
\begin{proposition}
    \label{not inner pluriclosed}
    Suppose $\g$ is not inner. Then $(G,J)$ admits no left-invariant pluriclosed metrics.
\end{proposition}
\begin{proof}
    Thanks to symmetrisation (Proposition \ref{symmetrizing to right T invariant}), it suffices to show that there are no left-invariant right $T$-invariant pluriclosed metrics on $(G,J)$. For the sake of contradiction, suppose $g$ is such a metric,  and $\omega$ is the corresponding positive $(1,1)$-form. Since $\g$ is not inner, there exists a complex root  $\alpha \in \Phi^+ \cap \RR$ (Proposition \ref{characterisation of type prop}). Fix a non-zero root vector $E_\alpha \in \g^\C_\alpha$, and let $H \in \l$. Then
    \begin{align*}
        dd^c \omega( H, \sigma H, E_\alpha, \sigma E_\alpha) &= -d^c \omega([H, \sigma H], E_\alpha, \sigma E_\alpha)  + d^c \omega([H, E_\alpha], \sigma H, \sigma E_\alpha) \\
        & \quad  - d^c \omega([H, \sigma E_\alpha], \sigma H, E_\alpha) -d^c\omega( [ \sigma H, E_\alpha], H, \sigma E_\alpha) \\   &\quad + d^c \omega([ \sigma H, \sigma E_\alpha], H, E_\alpha) - d^c \omega([E_\alpha, \sigma E_\alpha], H, \sigma H)  \\
        &=  - (\alpha + \sigma \alpha)(H) d^c \omega( \sigma H, E_\alpha, \sigma E_\alpha) \\
        &\quad + (\alpha + \sigma \alpha)(\sigma H) d^c \omega(  H, E_\alpha, \sigma E_\alpha) \\
        &= \i (\alpha + \sigma \alpha)(H) d \omega (\sigma H, E_\alpha, \sigma E_\alpha) \\
        & \quad + \i (\alpha + \sigma \alpha)(\sigma H) d \omega (H, E_\alpha, \sigma E_\alpha).
    \end{align*}
    Note that $[E_\alpha, \sigma E_\alpha] =0$, because $\alpha + \sigma \alpha$ is non-zero and not a root, since $\h$ is maximally compact.
    Now, given $H' \in \h^\C$, we find 
    \begin{align*}
        (d \omega)(H', E_\alpha, \sigma E_\alpha) &= - \omega([H', E_\alpha], \sigma E_\alpha) + \omega([H', \sigma E_\alpha], E_\alpha)-\omega ([E_\alpha, \sigma E_\alpha], H') \\
        &= -(\alpha + \sigma \alpha)(H') \omega(E_\alpha, \sigma E_\alpha) \\
        &= - \i (\alpha + \sigma \alpha)(H') g(E_\alpha, \sigma E_\alpha).
    \end{align*}
    Therefore, 
    \begin{align*}
        dd^c \omega ( H, \sigma H, E_\alpha, \sigma E_\alpha) &= 2 (\alpha + \sigma \alpha)(H) (\alpha + \sigma \alpha)( \sigma H) g(E_\alpha, \sigma E_\alpha) \\
        &= 2 | (\alpha + \sigma \alpha)(H)|^2 g(E_\alpha, \sigma E_\alpha),
    \end{align*}
    where the last equality holds because $(\alpha + \sigma \alpha)( \sigma H) = \overline{(\alpha + \sigma \alpha)(  H)}$. Since $\alpha + \sigma \alpha$ is non-zero, there exists $H \in \l$ such that $(\alpha + \sigma \alpha)(H) \neq 0$. In particular, the computation above shows that $ dd^c \omega ( H, \sigma H, E_\alpha, \sigma E_\alpha) \neq 0$, so $\omega$ is not pluriclosed.
\end{proof}

It remains to consider the inner case, so assume from now on that $g$ is inner. Let $g$ be a left-invariant right $T$-invariant Hermitian metric on $(G,J)$, and let $\omega = g(J \cdot, \cdot)$ be the corresponding positive $(1,1)$-form. Let $E_\alpha \in \g^\C_\alpha$ be root vectors chosen as in Lemma \ref{sigma adapted root vectors lemma}, and let $y_\alpha := \omega(E_\alpha, E_{-\alpha})$ for all $\alpha \in \Phi$.
 
Decompose $\g$ into its simple factors, $\g = \g_1 \oplus \cdots \oplus \g_s$, let $B_i$ denote the Killing form of $\g_i$ and let $\Phi = \Phi_1 \sqcup \cdots \sqcup \Phi_s$ be the corresponding decomposition of the root system. 
 Let $\alpha_1,\ldots, \alpha_r$ denote the simple roots of $\Phi^+$. For each $\alpha \in \Phi$, let $c_k(\alpha)$ be the integers given by $\alpha = \sum_{k=1}^r c_k(\alpha) \alpha_k$. Let $\V$ denote the Vogan diagram induced by $\Phi^+$. Let $\theta$ be the unique Cartan involution of $\g$ with $\theta \h = \h$, and set $\theta_i := \theta|_{\g_i} : \g_i \rightarrow \g_i$.

With a slight modification to the proof, \cite[Theorem 3.7]{lauret} (the compact case) generalises to the following result:
\begin{proposition}
    \label{pluriclosed metrics description}
    Suppose $\g$ is inner, and $g$ is left-invariant and right $T$-invariant. The following are equivalent:
    \begin{enumerate}[\normalfont (a)]
        \item $g$ is pluriclosed.
        \item There exist positive real numbers $\kappa_1,\ldots, \kappa_s$ such that $g|_{\h \times \h} =  -(\sum_{i=1}^s \kappa_i B_i)|_{\h \times \h}$, and for any $\alpha, \beta \in \Phi^+_i$ such that $\alpha + \beta$ is a root, we have
        \begin{equation}
        \label{pluriclosed roots equation 2}
            \i y_{\alpha +\beta} =  \i y_\alpha + \i y_\beta - \kappa_i.
        \end{equation}
    \end{enumerate}
\end{proposition}

Let $\lambda_\alpha = g(E_\alpha, \sigma E_\alpha)$, and let $s_\alpha = 1$ if $\alpha$ is compact and $s_\alpha = -1$ if $\alpha$ is non-compact. By the choice of $E_\alpha$ in Lemma \ref{sigma adapted root vectors lemma}, for $\alpha \in \Phi^+$, we find
$$\i y_\alpha = \i \omega(E_\alpha, E_{-\alpha}) = -g(E_\alpha,  E_{-\alpha}) = s_\alpha g(E_\alpha, - s_\alpha E_{-\alpha}) = s_\alpha \lambda_\alpha.$$
Therefore, Equation (\ref{pluriclosed roots equation 2}) is equivalent to 
\begin{equation}
    \label{pluriclosed roots equation 3}
    s_{\alpha + \beta} \lambda_{\alpha + \beta} = s_\alpha \lambda_\alpha + s_\beta \lambda_\beta - \kappa_i.
\end{equation}

\begin{proof}[Proof of Theorem \ref{pluriclosed technical theorem} Part (b)]
    First, let us show (i) implies (ii). Suppose $(G,J)$ admits a left-invariant pluriclosed metric $g$.  Now,  by symmetrising (Proposition \ref{symmetrizing to right T invariant}), we may assume that $g$ is right $T$-invariant, so  $g$ satisfies the conditions of Proposition \ref{pluriclosed metrics description} (b). Since $g|_{\h \times \h} = - (\sum \kappa_i B_i)|_{\h \times \h}$ and $g$ is Hermitian, we know that $- (\sum \kappa_i B_i)|_{\h \times \h}$ is compatible with $J|_\h$. Next, for the sake of contradiction, suppose that there is a connected component of $\V$ which is not listed in Table \ref{vogan table}.   By Proposition \ref{table 1 characterisation prop}, there exist non-compact positive roots $\alpha, \beta \in \Phi^+_i$ such that $\alpha + \beta$ is a root. Therefore, Equation (\ref{pluriclosed roots equation 3}) and Corollary \ref{compactness addition rules} imply that $0 < \lambda_{\alpha + \beta} = - \lambda_\alpha - \lambda_\beta - \kappa_i < 0$, a contradiction.

    Next, let us show that (ii) implies (iii). By (ii), we may choose $\kappa_1,\ldots, \kappa_s > 0$ such that $ - (\sum \kappa_i B_i)|_{\h \times \h}$ is $J|_\h$-invariant. Set $g := -\sum_i \kappa_i B_i (\cdot, \theta_i \cdot)$. Then $g$ is a left-invariant right $T$-invariant Riemannian metric with  $g|_{\h \times \h}  = - (\sum \kappa_i B_i)|_{\h \times \h}$. To see that $g$ is $J$-Hermitian, it suffices to show that $g(\q,\q) = 0$, where $\q = \l \oplus \bigoplus_{\alpha \in \Phi^+} \g^\C_\alpha$. Since $- (\sum \kappa_i B_i)|_{\h \times \h}$ is $J|_\h$-invariant, it follows that $g(\l,\l) = 0$.   Proposition \ref{right T invariant metrics form} implies that $g(\bigoplus_{\alpha \in \Phi}\g^\C_\alpha, \l) = 0$ and $g(\g^\C_\alpha, \g^\C_\beta) = 0$ for $\alpha, \beta \in \Phi^+$. Let us show that $g$ is pluriclosed. Observe that for $\alpha \in \Phi^+_i$, we have $$\lambda_\alpha = g(E_\alpha, \sigma E_\alpha) = -s_\alpha g(E_\alpha, E_{-\alpha}) =  s_\alpha \kappa_i B(E_\alpha, \theta E_{-\alpha}) = s_\alpha s_{-\alpha} \kappa_i B(E_\alpha, E_{-\alpha}) = \kappa_i.$$ Therefore, by Proposition \ref{pluriclosed metrics description} and Equation (\ref{pluriclosed roots equation 3}), it suffices to show that $$s_{\alpha + \beta}  = s_\alpha + s_\beta - 1$$ whenever $\alpha, \beta, \alpha + \beta \in \Phi^+$. Indeed, this equation holds because each connected component of $\V$ is listed in  Table \ref{vogan table}, so at most one of $\alpha$, $\beta$ is non-compact, by Proposition \ref{table 1 characterisation prop}.
\end{proof}

\bibliographystyle{alpha} 
\bibliography{references}
\end{document}